\documentclass[reqno]{amsart}

\usepackage{amsmath}
\usepackage{amsfonts}
\usepackage{amssymb,enumerate}
\usepackage{amsthm}
\usepackage[all]{xy}
\usepackage{rotating}
\usepackage{hyperref}

\theoremstyle{plain}
\newtheorem{lem}{Lemma}[section]

\newtheorem{prop}[lem]{Proposition}

\newtheorem{intthm}{Theorem}

\theoremstyle{definition}
\newtheorem{defn}[lem]{Definition}
\newtheorem{ex}[lem]{Example}

\newtheorem{disc}[lem]{Remark}

\newtheorem{fact}[lem]{Fact}
\newtheorem{para}[lem]{}

\newtheorem*{convention*}{Convention}

\newcommand{\cat}[1]{\mathcal{#1}}

\newcommand{\catd}{\cat{D}}

\newcommand{\pd}{\operatorname{pd}}	
\newcommand{\gdim}{\mathrm{G}\text{-}\!\dim}

\newcommand{\id}{\operatorname{id}}	
\newcommand{\fd}{\operatorname{fd}}

\newcommand{\cmdim}{\mathrm{CM}\text{-}\!\dim}	
\newcommand{\cidim}{\mathrm{CI}\text{-}\!\dim}

\newcommand{\depth}{\operatorname{depth}}	
\newcommand{\rank}{\operatorname{rank}}	
\newcommand{\amp}{\operatorname{amp}}
\newcommand{\edim}{\operatorname{edim}}

\newcommand{\grade}{\operatorname{grade}}

\newcommand{\cmd}{\operatorname{cmd}}	
\newcommand{\qcmd}{\mathbf{q}\!\operatorname{cmd}}

\newcommand{\lotimes}{\otimes^{\mathbf{L}}}
\newcommand{\HH}{\operatorname{H}}

\newcommand{\spec}{\operatorname{Spec}}

\newcommand{\Ker}{\operatorname{Ker}}

\newcommand{\ideal}[1]{\mathfrak{#1}}
\newcommand{\m}{\ideal{m}}
\newcommand{\n}{\ideal{n}}

\newcommand{\fm}{\ideal{m}}
\newcommand{\fn}{\ideal{n}}

\newcommand{\fr}{\ideal{r}}

\newcommand{\wt}{\widetilde}

\newcommand{\comp}[1]{\widehat{#1}}
\newcommand{\ol}{\overline}
\newcommand{\wti}{\widetilde}

\newcommand{\bbz}{\mathbb{Z}}

\newcommand{\from}{\leftarrow}
\newcommand{\xra}{\xrightarrow}
\newcommand{\xla}{\xleftarrow}

\newcommand{\vf}{\varphi}

\newcommand{\y}{\mathbf{y}}

\newcommand{\x}{\mathbf{x}}

\renewcommand{\geq}{\geqslant}
\renewcommand{\leq}{\leqslant}
\renewcommand{\ker}{\Ker}

\newcommand{\Ext}[4][R]{\operatorname{Ext}_{#1}^{#2}(#3,#4)}	
\newcommand{\Rhom}[3][R]{\mathbf{R}\!\operatorname{Hom}_{#1}(#2,#3)}	
\newcommand{\Lotimes}[3][R]{#2\otimes^{\mathbf{L}}_{#1}#3}
\newcommand{\Otimes}[3][R]{#2\otimes_{#1}#3}

\newcommand{\Comp}[2]{\widehat{#1}^{\ideal{#2}}}

\numberwithin{equation}{lem}

\newcommand{\power}[2]{#1[\![#2]\!]}

\begin{document}

\bibliographystyle{amsplain}

\title{Cohen Factorizations: Weak Functoriality and Applications}

\dedicatory{To Hans-Bj\o rn Foxby on the occasion of his sixty-fifth birthday}

\author{Saeed Nasseh}
\author{Sean Sather-Wagstaff}

\thanks{Sather-Wagstaff  was supported in part by North Dakota EPSCoR, 
National Science Foundation Grant EPS-0814442,
and  a grant from the NSA}

\address{Department of Mathematics,
North Dakota State University Dept \# 2750,
PO Box 6050,
Fargo, ND 58108-6050
USA}

\email{saeed.nasseh@ndsu.edu}
\urladdr{https://www.ndsu.edu/pubweb/~{}nasseh/}

\email{sean.sather-wagstaff@ndsu.edu}
\urladdr{http://www.ndsu.edu/pubweb/\~{}ssatherw/}



\begin{abstract}
We investigate Cohen factorizations of local ring homomorphisms from
three perspectives. First, we prove a ``weak functoriality'' result for
Cohen factorizations: certain morphisms of local ring homomorphisms
induce morphisms of Cohen factorizations. 
Second, we use Cohen factorizations to study 
the  properties of local ring homomorphisms
(Gorenstein, Cohen-Macaulay, etc.) in certain commutative diagrams.
Third, we use Cohen factorizations to investigate the structure of
quasi-deformations of local rings, with an eye on the question of
the behavior of CI-dimension in short exact sequences.
\end{abstract}

\maketitle

\section{Introduction} \label{sec0}

\begin{convention*}
Throughout this paper, the term ``ring'' is short for 
``commutative noetherian ring with identity''. 
A local ring is ``complete'' when it is complete with respect to its maximal ideal.
Let $R$, $S$, and $T$ be rings.
\end{convention*}

Cohen factorizations were introduced in~\cite{avramov:solh}
as tools to study local ring homomorphisms. 
(See Section~\ref{sec120417b} for definitions and background material.)
The utility of these factorizations
can be seen in their many applications; see, e.g.,~\cite{avramov:lcih,
avramov:glh, avramov:lgh, avramov:rhafgd, avramov:cmporh, frankild:qcmpolh,
iyengar:golh, sather:cidfc}. The main point of this construction is that
it allows one to study a local ring homomorphism  by 
replacing it with a surjective one; thus, one can  assume that
the target is finitely generated over the source, so one can apply finite homological
algebra techniques. 

In Section~\ref{sec120426a} of this paper, we investigate  functorial
properties of Cohen factorizations. The main result of this section is the following;
its proof is in~\ref{proof120417a}.
Example~\ref{ex130327a} shows that the separability assumptions~\eqref{intthm120201a3} are necessary.

\begin{intthm}\label{intthm120201a}
Consider a commutative diagram of local ring homomorphisms
\begin{equation}\label{eq120222a}\tag{A.1}
\begin{split}
\xymatrix{
(R,\fm)\ar[r]^{\varphi}\ar[d]_{\alpha}&(S,\fn)\ar[d]^{\beta}\\
(\widetilde{R},\widetilde{\fm})\ar[r]^{\widetilde{\varphi}}&(\widetilde{S},\widetilde{\fn})
}
\end{split}\end{equation}
with the following properties:
\begin{enumerate}[\rm(1)]
\item\label{intthm120201a2}
$\alpha$ and $\vf$ have regular factorizations
$R\xra{\dot{\alpha}} R''\xra{\alpha'} \wti R$
and $R\xra{\dot{\varphi}} R'\xra{\varphi'} S$,
e.g., $\wti R$ and $S$ are complete,
\item\label{intthm120201a3}
$\widetilde{S}$ is complete, and
the  field extensions 
$R/\m\to \wti R/\wti \m\to\widetilde{S}/\widetilde{\fn}$  
are separable.
\end{enumerate}
Let
$S\xra{\dot{\beta}} S'\xra{\beta'} \wti S$
and 
$\widetilde{R}\xra{\dot{\wti\vf}} \wti R'\xra{\wti\vf'} \widetilde{S}$
be   Cohen factorizations
of $\beta$ and $\widetilde{\varphi}$. 
Then there is a commutative diagram of local ring
homomorphisms 
\begin{equation}\label{eq120222z}\tag{A.2}
\begin{split}
\xymatrix{
R\ar[r]^{\dot{\varphi}}\ar[d]_{\dot\alpha}&R'\ar[r]^{\varphi'}\ar[d]^{\dot\gamma}&S\ar[d]^{\dot\beta}\\
R''\ar[r]^{\dot{\sigma}}\ar[d]_{\alpha'}&T\ar[r]^{\sigma'}\ar[d]^{\gamma'}&S'\ar[d]^{\beta'}\\
\widetilde{R}\ar[r]^{\dot{\wti\vf}}&\wti{R}'\ar[r]^{\wti\vf'}&\widetilde{S}
}
\end{split}\end{equation}
such that the
diagrams $R'\xra{\dot\gamma}T\xra{\gamma'}\wti{R}'$
and $R''\xra{\dot\sigma}T\xra{\sigma'}S'$ are minimal Cohen factorizations.
\end{intthm}

We think of this as a result about functoriality of regular (e.g., Cohen) 
factorizations as follows.
The diagram~\eqref{eq120222a} is a  morphism in the category of 
local ring homomorphisms. Our result provides the following commutative
diagram of local ring homomorphisms where $\gamma=\gamma'\dot\gamma$
$$\xymatrix{
R\ar[r]^{\dot{\varphi}}\ar[d]_{\alpha}&R'\ar[r]^{\varphi'}\ar[d]^{\gamma}&S\ar[d]^{\beta}\\
\widetilde{R}\ar[r]^{\dot{\wti\vf}}&\wti{R}'\ar[r]^{\wti\vf'}&\widetilde{S}
}
$$
which is a morphism in the category of regular factorizations.
Of course, the operation that maps a local ring homomorphism to a
regular (or Cohen) factorization is not well-defined; hence our terminology ``weak 
functoriality''.

Given  a diagram~\eqref{eq120222a} where $\alpha$ and $\beta$ are ``nice'', 
the maps $\vf$ and $\wti\vf$ are intimately related. This 
maxim is the subject of Section~\ref{sec120306a},
which culminates in the proof of the next result.
It is one of the applications of Cohen factorizations 
mentioned in the title; see~\ref{proof120417c} for the proof.

\begin{intthm}
\label{prop120322a}
Consider a commutative diagram of local ring homomorphisms
$$
\xymatrix{
(R,\fm)\ar[r]^{\varphi}\ar[d]_{\alpha}&(S,\fn)\ar[d]^{\beta}\\
(\widetilde{R},\widetilde{\fm})\ar[r]^{\widetilde{\varphi}}&(\widetilde{S},\widetilde{\fn})
}
$$
such that $\alpha$ is  weakly Cohen,
$\beta$ is  weakly regular, and
the induced map $\widetilde{R}/\widetilde{\fm}\to \widetilde{S}/\widetilde{\fn}$ is separable. Let P be one of the following conditions:
Gorenstein, quasi-Gorenstein, complete intersection, 
Cohen-Macaulay, quasi-Cohen-Macaulay.
Then $\vf$ is P if and only if $\wti\vf$ is P.
\end{intthm}

Section~\ref{sec120306a} also contains, among other things, a
base change result for CI-dimension and CM-dimension
(Proposition~\ref{prop120312e}) which may be
of independent interest. 

The paper concludes with Section~\ref{sec120417a}, which is devoted 
almost entirely to the proof of
the next result; see~\ref{proof120417b}.
It is another application of Cohen factorizations. 
Also, it is related to the short exact sequence question for
CI-dimension, as we describe in Remark~\ref{disc120426a}.

\begin{intthm}\label{intthm120201d}
For $i=1,2$
let $(R,\m,k)\xra{\vf_i}(R_i,\m_i,k_i)\xla{\tau_i}(Q_i,\n_i,k_i)$ be a quasi-deformation
such that the field extension $k\to k_i$ is separable.
\begin{enumerate}[\rm(a)]
\item\label{intthm120201d01}
Then there exists a 
commutative diagram of local ring homomorphisms
$$
\xymatrix@C=7mm{
(Q_1,\n_1,k_1)\ar[r]^-{\tau_1}\ar[dr]_-{\ol\delta_1\delta_1}&(R_1,\m_1,k_1)\ar[rd]_-{\sigma_1\beta_1}&(R,\m,k)\ar[r]^-{\vf_2}\ar[l]_-{\vf_1}\ar[d]^-{\alpha'}
&(R_2,\m_2,k_2)\ar[ld]^-{\sigma_2\beta_2}&(Q_2,\n_2,k_2)\ar[l]_-{\tau_2}\ar[dl]^-{\ol\delta_2\delta_2}\\
&(\overline{Q}_1,\ol{\n}_1,k')\ar[r]_-{\ol\gamma_1}&(R',\m',k')&(\overline{Q}_2,\ol{\n}_1,k')\ar[l]^-{\ol\gamma_2}& \hspace{11mm}\text{\emph{(C.1)}}
}
$$
such that $\alpha'$ is flat, each map $\ol\delta_i\delta_i$ is weakly Cohen, and
each map $\ol\gamma_i$ is surjective.
(See the proof for an explanation of the labels $\delta_i$, $\sigma_i$, etc.)
\item\label{intthm120201d02}
Assume that each map $\vf_i$ is weakly Cohen. 
Then the maps in  diagram~\emph{(C.1)} satisfy the following properties for $i=1,2$:
\begin{enumerate}[\rm(b1)]
\item\label{intthm120201d1}
The maps
$R_i\to R'$ are weakly Cohen.
\item\label{intthm120201d2}
The diagram $R\to R'\twoheadleftarrow \overline{Q}_i$ is a  quasi-deformation
and each parallelogram diagram is a pushout.
\item\label{intthm120201d4}
Given  a  $R$-module $M$, if
$\pd_{Q_i}(M\otimes_RR_i)<\infty$, then $\pd_{\overline{Q}_i}(M\otimes_RR')<\infty$.
\end{enumerate}
\end{enumerate}
\end{intthm}

As the list of references indicates, this work is built on ideas pioneered
by Hans-Bj\o rn Foxby, his collaborators, and his students. 
We are deeply grateful to him for his many contributions
to this field through his research and his support.

\section{Background}
\label{sec120417b}

This section lists important foundational material for use in the rest of the paper.

\subsection*{Factorizations}

\begin{defn}\label{weakly reg}
Let $\vf\colon (R,\fm)\to (S,\fn)$ be a  local ring homomorphism.
\begin{enumerate}[(a)]
\item
The \emph{embedding dimension} of $\vf$ is $\edim(\vf)=\edim(S/\m S)$.
\item\label{weakly reg1}
The map $\vf$ is \emph{weakly regular}
if it is flat and the closed fiber $S/\fm S$ is  regular.
\item
The map $\vf$ is a \emph{weakly Cohen} if 
it is weakly regular, 
and the induced field extension $R/\m\to S/\n$ is separable, e.g., if $\operatorname{char}(R/\m)=0$ or $R/\m$ is perfect of positive characteristic.
\item
The map $\vf$ is  \emph{Cohen} if 
it is weakly Cohen such that
$\n=\m S$, that is, it is weakly Cohen such that the closed fiber is a field.
\end{enumerate}
\end{defn}

\begin{fact}\label{fact120215a}
Let $R\xra{\vf} S\xra{\psi} T$ be weakly regular local ring homomorphisms.
Then the composition $\psi\vf$ is weakly regular such that
$\edim(\psi\vf)=\edim(\psi)+\edim(\vf)$;
see~\cite[5.9]{iyengar:golh}.
\end{fact}

The central objects of study for this paper, 
defined next, are from the work of Avramov, Foxby, and B. Herzog~\cite{avramov:solh}.

\begin{defn}\label{weakly reg'}
Let $\vf\colon (R,\fm)\to (S,\fn)$ be a  local ring homomorphism.
\begin{enumerate}[(a)]
\item\label{weakly reg2}
A \emph{regular  factorization} of $\vf$
is a diagram of local  homomorphisms
$R\xra{\dot{\varphi}} R'\xra{\varphi'} S$
such that $\varphi=\varphi'\dot{\varphi}$, the map $\dot{\varphi}$ is 
weakly regular, and $\varphi'$ is surjective.
\item\label{weakly reg3}
A
\emph{Cohen factorization} of $\vf$ is a regular factorization
$R\xra{\dot{\varphi}} R'\xra{\varphi'} S$ of $\vf$ such that $R'$ is complete.
\item\label{weakly reg4}
A \emph{comparison} of one Cohen factorization
$R\xra{\dot{\varphi}} R'\xra{\varphi'} S$ of $\vf$ to another one
$R\xra{\ddot{\varphi}} R''\xra{\varphi''} S$
 is a local homomorphism
$\upsilon\colon R'\to R''$ making the following diagram commute:
$$
\xymatrix{
&R'\ar[rd]^{\varphi'}\ar[d]_{\upsilon}&\\
R\ar[ur]^{\dot{\varphi}}\ar[r]^{\ddot{\varphi}}&
R''\ar[r]^{\varphi''}& S.
}
$$
\item
The \emph{semicompletion} of $\vf$, denoted $\grave\vf\colon R\to \comp S$, 
is the composition of $\vf$ with the natural map $S\to \comp S$.
\end{enumerate}
\end{defn}

\begin{disc}\label{remark minimal cf}
Let $R\xra{\dot{\varphi}} R'\xra{\varphi'} S$ be a regular factorization of
a local ring homomorphism
$\varphi\colon R\to S$. The surjective
homomorphism $\varphi'\colon R'\to S$ induces a
surjective homomorphism
$R'/\fm R'\to S/\fm S$
which implies that
$\edim(\dot\vf)\geq \edim(\vf)$.
Since $\dot{\varphi}$ is weakly regular, we conclude that
$$
\dim(R')-\dim(R)=\edim(\dot\vf)\geq \edim(\vf).
$$
\end{disc}

\begin{defn}\label{minimal cf}
A regular factorization
$R\xra{\dot{\varphi}} R'\xra{\varphi'} S$
of
a local ring homomorphism
$\varphi\colon R\to S$ is  \emph{minimal} if $\dim(R')-\dim(R)= \edim(\vf)$.
\end{defn}

\begin{fact}\label{fact120418a}
Let $\vf\colon (R,\m)\to (S,\n)$ and 
$\psi\colon S\to T$ be  local ring homomorphisms.

\begin{enumerate}[(a)]
\item \label{fact120418a1}
By~\cite[(1.1) Theorem and (1.5) Proposition]{avramov:solh}, if $S$ is complete, then $\vf$ has a minimal Cohen factorization.
Since $\comp S$ is complete, it follows that the semicompletion  $\grave\vf\colon R\to \comp S$ 
has a minimal Cohen factorization.

\item \label{fact120418a2}
If $\vf$ is surjective, $\psi$ is weakly regular, and $T$ is complete,
then in any minimal Cohen factorization 
$R\to R'\to T$ of the composition $\psi\vf$, we have
$T\cong R'\otimes_R S$, that is, the diagram
$$
\xymatrix{
R\ar[r] \ar[d] & R' \ar[d] \\
S \ar[r] & T}$$
is a pushout; see the proof of~\cite[(1.6) Theorem]{avramov:solh}.
\item \label{fact120418a3}
Assume that $S$ is complete, and consider two Cohen factorizations
$R\to R'\to S$ and $R\to R''\to S$ of $\vf$.
If the extension $R/\m\to S/\n$ is separable, then 
there is a comparison $\upsilon\colon R'\to R''$ of the first factorization
to the second one;
moreover, if both Cohen factorizations are minimal, then any comparison between them is an isomorphism.
See~\cite[(1.7) Proposition]{avramov:solh}.
\end{enumerate}
\end{fact}

\subsection*{Homological Notions}

\begin{defn}
The flat dimension of a local ring homomorphism $\vf\colon R\to S$ is
$\fd(\vf)=\fd_R(S)$.
\end{defn}

The remaining aspects of this subsection use 
the derived category, so we specify some notations.
References on the subject 
include~\cite{gelfand:moha, hartshorne:rad, verdier:cd, verdier:1}.

\begin{defn}
The derived category of the category of  $R$-modules
is denoted $\catd(R)$. The objects of this category are the $R$-complexes
which we index homologically:
$X=\cdots\to X_i\to X_{i-1}\to\cdots$.

Let $X$ and $Y$ be $R$-complexes.
We say that $X$ is \emph{homologically bounded}
if 
$\HH_i(X)=0$ for $|i|\gg 0$, 
and $X$ is \emph{homologically finite}
if the total homology module
$\HH(X)=\oplus_{i\in\bbz}\HH_i(X)$ is finitely generated over $R$.
The \emph{supremum}, \emph{infimum}, and \emph{amplitude}
of $X$ are
\begin{align*}
\sup(X)
&=\sup\{i\in\bbz\mid\HH_i(X)\neq 0\} \\
\inf(X)
&=\inf\{i\in\bbz\mid\HH_i(X)\neq 0\} \\
\amp(X)&=\sup(X)-\inf(X).
\end{align*}
The derived functors of Hom and tensor product
are denoted $\Rhom --$ and $\Lotimes --$.
For each $i\in\bbz$, set $\Ext iXY:=\HH_{-i}(\Rhom XY)$.
\end{defn}

The next definition originates with work of Auslander and 
Bridger~\cite{auslander:adgeteac,auslander:smt}.

\begin{defn}
Let  $X$ be a homologically finite $R$-complex.
Then  $X$ is \emph{derived reflexive} if
$\Rhom XR$ is homologically finite
(that is,  homologically bounded) and the natural biduality map
$X\to\Rhom{\Rhom XR}R$ is an isomorphism in $\catd(R)$.
The \emph{G-dimension} of $X$ is
$$\gdim_R(X):=
\begin{cases}
-\inf(\Rhom XR) & \text{if $X$ is derived reflexive over $R$} \\
\infty & \text{otherwise.}
\end{cases}$$
\end{defn}

\begin{disc}
For a finitely generated $R$-module, the G-dimension defined above
is the same as the one defined by Auslander and Bridger;
see~\cite[2.7. Theorem]{yassemi:gd}.
More generally, the G-dimension of a homologically finite $R$-complex 
has a similar interpretation by
Christensen~\cite[(2.3.8) GD Corollary]{christensen:gd}.
\end{disc}

Given a finitely generated $R$-module $M$,
the fact that $\Ext iMR=0$ for all $i>\gdim_R(M)$ implies
that $\grade_R(M)\leq\gdim_R(M)$. This motivates the next definition.

\begin{defn}
A finitely generated $R$-module $M$ is 
\emph{G-perfect} if $\grade_R(M)=\gdim_R(M)$.
\end{defn}

We proceed with CI-dimension of Avramov, Gasharov, and Peeva~\cite{avramov:cid, sather:cidc, sather:cidfc} and Gerko's
CM-dimension~\cite{gerko:ohd}.

\begin{defn} \label{cidim01}
Consider a diagram $R\xra{\vf} R'\xla{\tau} Q$
of local ring homomorphisms  such that $\vf$ is
flat, and $\tau$ is surjective.
Such a diagram is a 
\emph{G-quasi-deformation}  if $R'$ is  G-perfect as a $Q$-module.
Such a diagram is a 
\emph{quasi-deformation}  if $\Ker(\tau)$ is 
generated by a $Q$-regular sequence.
\end{defn}

\begin{defn} \label{cidim03}
Let $X$ be a homologically finite complex over a local ring $R$.
The 
\emph{CM-dimension} 
and 
\emph{CI-dimension}
of $X$ are
\begin{align*}
\cmdim_R(X)
&:=\inf\left\{\gdim_Q(R'\lotimes_R X)-\gdim_Q(R')\left| \text{
\begin{tabular}{@{}c@{}}
$R\to R'\from Q$ is a \\ G-quasi-deformation 
\end{tabular}
}\!\!\!\right. \right\}\\
\cidim_R(X)
&:=\inf\left\{\pd_Q(R'\lotimes_R X)-\pd_Q(R')\left| \text{
\begin{tabular}{@{}c@{}}
$R\to R'\from Q$ is a \\ quasi-deformation 
\end{tabular}
}\!\!\!\right. \right\}. 
\end{align*}
\end{defn}

\begin{fact}\label{fact120312a}
Let $\alpha'\colon(R',\m',k')\to (\wti R',\wti \m',\wti k')$ be a flat local
ring homomorphism, 
and let $M$ be a homologically finite
$R'$-complex. 
The proof of~\cite[(1.11) Proposition]{avramov:cid}
shows that
$\cidim_{\wti R'}(\Lotimes[R']{\wti R'}{M})\geq\cidim_{R'}(M)$
with equality holding when $\cidim_{\wti R'}(\Lotimes[R']{\wti R'}{M})<\infty$.
A similar argument shows that
$\cmdim_{\wti R'}(\Lotimes[R']{\wti R'}{M})\geq\cmdim_{R'}(M)$
with equality holding when $\cmdim_{\wti R'}(\Lotimes[R']{\wti R'}{M})<\infty$.
\end{fact}

\begin{defn}
Assume that $(R,\m,k)$ is local, and let $X$ be a homologically finite $R$-complex.
For  $i\in\bbz$, the $i$th
\emph{Bass number} and \emph{Betti number} of $X$ are
$\mu^i_R(X)=\rank_k(\Ext ikX)$
and
$\beta_i^R(X)=\rank_k(\Ext iXk)$.
The \emph{Bass series} and \emph{Poincar\'e series} of $X$ are
the formal Laurent series
$I^X_R(t)=\sum_{i\in\bbz}\mu^i_R(X)t^i$
and
$P_X^R(t)=\sum_{i\in\bbz}\beta_i^R(X)t^i$. 
\end{defn}

\begin{defn}
A homologically finite $R$-complex $C$ is \emph{semidualizing} for $R$
if the natural morphism $R\to\Rhom CC$ is an isomorphism in $\catd(R)$.
A \emph{dualizing complex} for $R$ is a semidualizing complex
$D$ of finite injective dimension, i.e., 
such that $D$ is isomorphic in $\catd(R)$
to a bounded complex of injective $R$-modules.
\end{defn}

\begin{disc}
Assume that $R$ is local. Then $R$ has a dualizing complex if and only if
it is a homomorphic image of a Gorenstein local ring;
one implication is from Grothendieck and Hartshorne~\cite{hartshorne:rad},
and the other is by Kawasaki~\cite{kawasaki:mns}. In particular, a complete
local ring has a dualizing complex.
\end{disc}

\subsection*{Properties of Ring Homomorphisms}

\

\noindent
The first notion in this subsection is from Avramov~\cite{avramov:lcih}.

\begin{defn}
Let $\vf\colon R\to S$ be a local ring homomorphism. Given a Cohen factorization $R\xra{\dot\vf} R'\xra{\vf'} \comp S$ of the semicompletion
$\grave\vf\colon R\to\comp S$, we say that
$\vf$ is \emph{complete intersection} if
$\Ker(\vf')$ is generated by an $R'$-regular sequence.
\end{defn}

\begin{disc}
Let $\vf\colon R\to S$ be a local ring homomorphism.
The complete intersection property for $\vf$ is independent of the
choice of Cohen factorization by~\cite[(3.3) Remark]{avramov:lcih}. 
Also $R$ and $\vf$ are complete intersection if and only if
$S$ is complete intersection and $\fd(\vf)$ is finite;
see~\cite[(5.9), (5.10), and (5.12)]{avramov:lcih}.
\end{disc}

The next notion is mostly due to
Avramov and Foxby~\cite{avramov:rhafgd},
with some contributions from  Iyengar and Sather-Wagstaff~\cite{iyengar:golh}.

\begin{defn}
Let $\vf\colon R\to S$ be a local ring homomorphism.
Given a Cohen factorization $R\xra{\dot\vf} R'\xra{\vf'} \comp S$ of the semicompletion
$\grave\vf\colon R\to\comp S$, we set
$$\gdim(\vf):=
\gdim_{R'}(\comp S)-\edim(\dot\vf). $$
\end{defn}

\begin{disc}
The G-dimension of a local homomorphism is independent of the choice
of Cohen factorization by~\cite[3.2.~Theorem]{iyengar:golh}.
\end{disc}

\begin{defn}
Let $\vf\colon R\to S$ be a local ring homomorphism,
and let $D^{\comp R}$ be a dualizing complex for $\comp R$.
A \emph{dualizing complex} for $\vf$ is a semidualizing $S$-complex
$D^{\vf}$ such that $\Lotimes[\comp R]{D^{\comp R}}{(\Lotimes[S]{\comp S}{D^{\vf}})}$
is a dualizing complex for $\comp S$.
\end{defn}

\begin{fact}
Let $\vf\colon R\to S$ be a local ring homomorphism of finite G-dimension,
e.g., finite flat dimension; see~\cite[(4.4.2)]{avramov:rhafgd}.
If $S$ is complete, then $\vf$ has a dualizing complex 
by~\cite[(6.7) Lemma]{avramov:rhafgd};
specifically, given a Cohen factorization $R\to R'\to S$ of $\vf$,
the $S$-complex $\Rhom[R']{S}{R'}$ is dualizing for $\vf$.
In particular, the semicompletion   $\grave\vf\colon R\to \comp S$
has a dualizing complex. 
\end{fact}

Next are some notions of Avramov and Foxby~\cite{avramov:rhafgd,avramov:cmporh}
and Frankild~\cite{frankild:qcmpolh}.

\begin{defn}
Let $\vf\colon R\to S$ be a local ring homomorphism  of finite G-dimension
and let $D^{\grave \vf}$ be a dualizing complex for 
the semicompletion   $\grave\vf\colon R\to \comp S$.
The \emph{quasi-Cohen-Macaulay defect} of $\vf$
is $\qcmd(\vf):=\amp(D^{\grave \vf})$, and
$\vf$ is \emph{quasi-Cohen-Macaulay} if $\qcmd(\vf)=0$,
that is, if $D^{\grave \vf}$ is isomorphic in $\catd(\comp S)$ to 
a module. 

When $\fd(\vf)<\infty$,  the \emph{Cohen-Macaulay defect} of $\vf$
is $\cmd(\vf)=\qcmd(\vf)$; see~\cite[(5.5)]{avramov:rhafgd}.
The map $\vf$ is \emph{Cohen-Macaulay}
if it is quasi-Cohen-Macaulay and has finite flat dimension.
\end{defn}

\begin{disc}
Let $\vf\colon R\to S$ be a local ring homomorphism  of finite G-dimension.
The quasi-Cohen-Macaulay defect of $\vf$ 
is independent of the choice of dualizing complex,
as the dualizing complex is unique up to shift-isomorphism;
see~\cite[(5.4)]{avramov:rhafgd} and~\cite[(6.5)]{frankild:qcmpolh}.
If $R$ is Cohen-Macaulay and $\vf$ is quasi-Cohen-Macaulay,
then $S$ is Cohen-Macaulay; the converse holds when $\vf$ has finite flat dimension
or the induced map $\spec(\comp S)\to\spec(\comp R)$ is surjective
by~\cite[(3.10) Theorem]{avramov:solh} and~\cite[(7.7)]{frankild:qcmpolh}.
\end{disc}

\begin{defn}
Let $\vf\colon R\to S$ be a local ring homomorphism of finite G-dimension, and let
$D^{\grave\vf}$ be a dualizing complex for the semicompletion
$\grave\vf\colon R\to\comp S$ such that 
$\inf(D^{\grave\vf})=\depth(S)-\depth(R)$.
The \emph{Bass series} for $\vf$ is 
the Poincar\'e series $I_{\vf}(t):=P^{\comp S}_{D^{\grave\vf}}(t)$.
\end{defn}

\begin{disc}
Let $\vf\colon R\to S$ be a local ring homomorphism of finite G-dimension.
The Bass series for $\vf$ is a formal Laurent series $I_{\vf}(t)$
with non-negative integer coefficients satisfying the formal relation
$I^S_S(t)=I^R_R(t)I_{\vf}(t)$; see~\cite[(7.1) Theorem]{avramov:rhafgd}.
\end{disc}

\begin{defn}
Let $\vf\colon R\to S$ be a local ring homomorphism of finite G-dimension.
Then $\vf$ is \emph{quasi-Gorenstein} if $I_{\vf}(t)=t^{\depth(S)-\depth(R)}$,
that is, if $\comp S$ is  dualizing for $\grave\vf\colon R\to\comp S$.
The map $\vf$ is \emph{Gorenstein} if it is quasi-Gorenstein and has 
finite flat dimension.
\end{defn}

\begin{disc}
If $\vf\colon R\to S$ is a local  homomorphism with $\gdim(\vf)<\infty$,
then $S$ is Gorenstein if and only if $R$ is Gorenstein and $\vf$ is
quasi-Gorenstein
by~\cite[(7.7.2)]{avramov:rhafgd}.
\end{disc}

The CI-dimension of $\vf$ is from~\cite{sather:cidfc},
and the CM-dimension of $\vf$ works similarly.

\begin{defn}
Let $\vf\colon R\to S$ be a local ring homomorphism.
The \emph{CM-dimension of  $\vf$}  
and \emph{CI-dimension of  $\vf$}  
are
\begin{align*}
\cmdim(\vf)&:= 
\inf\left\{\cmdim_{R'}(\comp{S})-\edim(\Dot{\vf})
\left| \text{\begin{tabular}{c}
$R\xra{\dot\vf}R'\xra{\vf'}\comp{S}$ is a Cohen \\
factorization of $\grave\vf$
\end{tabular}}\right.\!\!\!\right\}\\
\cidim(\vf)&:= 
\inf\left\{\cidim_{R'}(\comp{S})-\edim(\Dot{\vf})
\left| \text{\begin{tabular}{c}
$R\xra{\dot\vf}R'\xra{\vf'}\comp{S}$ is a Cohen \\
factorization of $\grave\vf$
\end{tabular}}\right.\!\!\!\right\}.
\end{align*}
\end{defn}

\begin{disc}
We do not know whether the finiteness of
CM-dimension and/or CI-dimension of a local homomorphism 
is independent of the choice
of Cohen factorization.
\end{disc}

\section{Weak Functoriality of Regular/Cohen Factorizations}
\label{sec120426a}

The main objective of this section 
is the proof of Theorem~\ref{intthm120201a} from the introduction.
We begin with a lemma.

\begin{lem}\label{for Cohen fact 1}
Let $\vf\colon (R,\fm)\to (S,\fn)$ be a weakly regular local ring homomorphism. Then
$\edim(R)+\edim(S/\m S)=\edim(S)$.
\end{lem}

\begin{proof}
Set $e=\edim(R)$, and let $\x=x_1,\ldots,x_e\in\fm$ be a minimal generating sequence 
for $\m$.
Set $d=\edim(S/\m S)$, and let $\y=y_1,\ldots,y_d\in\fn$ be such that the 
residue sequence 
$\ol{\y}\in\n/\fm S$  generates  $\n/\fm S$ minimally over $S$
and over $S/\m S$.
It is straightforward to show that the concatenated  sequence $\vf(\x),\y\in \n$ generates 
$\n$.
It remains to show that this sequence generates $\n$ minimally.

Suppose by way of contradiction that the sequence $\vf(\x),\y$ does not generate $\n$ minimally.
The sequence $\vf(\x),\y$ contains a minimal generating sequence for $\n$.
Thus, either  one of the $\vf(x_i)$ or one of the $y_j$ is redundant
as part of this generating sequence.

Case 1: $y_j$ is redundant
as part of this generating sequence.
Assume without loss of generality that $j=1$.
In this case we have $\n=(\vf(x_1),\ldots,\vf(x_e),y_2,\ldots,y_d)S$.
Since $\m S=(\vf(\x))S$ it follows that the ideal
$\n/\m S$ is generated by the residue sequence $\ol{y_2},\ldots,\ol{y_d}$.
This contradicts the minimality of the original sequence $\y$.

Case 2: $\vf(x_i)$ is redundant as part of generating sequence for $\n$.
Assume without loss of generality that $i=1$.
It follows that the ideal $\n/(\y)S\subset S/(\y)S$ is generated by 
the residue sequence $\ol{\vf(x_2)},\ldots,\ol{\vf(x_e)}\in\n/(\y)S$.

The map $\vf$ is weakly regular, and the sequence $\ol{\y}$ minimally generates
the maximal ideal of the regular local ring $S/\m S$.
In particular, this sequence is $S/\m S$-regular,
so we know that the sequence $\y$ is $S$-regular and the induced map
$R\to S/(\y)S$ is flat; see, e.g.~\cite[Corollary to Theorem 22.5]{matsumura:crt}.
Since this map is flat and local, it follows from~\cite[(2.3) Lemma]{herzog:mlr}
that the ideal $\m (S/(\y)S)=\n/(\y)S$ is minimally generated by the 
sequence $\ol{\vf(x_1)},\ol{\vf(x_2)},\ldots,\ol{\vf(x_e)}$.
This contradicts the conclusion of the previous paragraph.
\end{proof}

Our next result will provide a vertical map in the proof of
Theorem~\ref{intthm120201a}.

\begin{prop}\label{thm120201a}
Consider a commutative diagram of local ring homomorphisms
$$
\xymatrix{
(R,\fm)\ar[r]^{\varphi}\ar[d]_{\alpha}&(S,\fn)\ar[d]^{\beta}\\
(\widetilde{R},\widetilde{\fm})\ar[r]^{\widetilde{\varphi}}&(\widetilde{S},\widetilde{\fn})
}
$$
such that $\alpha$ is  weakly regular,
$\widetilde{S}$ is complete, $\beta$ is  weakly regular, and
the induced map $R/\fm\to \widetilde{S}/\widetilde{\fn}$ is separable.
Assume that $\vf$ has a minimal regular factorization
$R\xra{\dot{\varphi}} R'\xra{\varphi'} S$,
and fix a minimal Cohen factorization
$\widetilde{R}\xra{\dot{\wti\vf}} S'\xra{\wti\vf'} \widetilde{S}$
of $\widetilde{\varphi}$. 
Then there is a weakly regular local ring
homomorphism $\alpha'\colon R'\to S'$ such that the
next diagram commutes
\begin{equation}\label{eq120209a}
\begin{split}
\xymatrix{
R\ar[r]^{\dot{\varphi}}\ar[d]_{\alpha}&R'\ar[r]^{\varphi'}\ar[d]^{\alpha'}&S\ar[d]^{\beta}\\
\widetilde{R}\ar[r]^{\dot{\wti\vf}}&S'\ar[r]^{\wti\vf'}&\widetilde{S}
}
\end{split}\end{equation}
and such that the second square is a pushout.
\end{prop}

\begin{proof} 
Fact~\ref{fact120418a}\eqref{fact120418a2}
provides a
commutative diagram of local ring homomorphisms
\begin{equation}\label{eq120208a}
\begin{split}
\xymatrix{
(R',\fm')\ar@{-->}[d]_{\beta'}\ar[r]^{\varphi'}&(S,\fn)\ar[d]^{\beta}\\
(S'',\fn'')\ar@{-->}[r]^{\rho''}&(\widetilde{S},\widetilde{\fn})
}\end{split}\end{equation}
where $R'\xra{\beta'} S''\xra{\rho''} \widetilde{S}$ is a minimal
Cohen factorization for $\beta \varphi'$ and the diagram is a pushout.
Fact~\ref{fact120215a} implies  that the diagrams
$R\xra{\dot{\wti\vf}\alpha} S'\xra{\wti\vf'}\widetilde{S}$
and $R\xra{\beta'\dot{\varphi}} S''\xra{\rho''} \widetilde{S}$
are Cohen factorizations of $\beta\varphi$. 
In particular, the next diagram commutes:
$$
\xymatrix{
R\ar[r]^{\dot{\wti\vf}\alpha}\ar[d]_{\beta'\dot{\varphi}}
\ar[rd]^{\beta\varphi}&S'\ar[d]^-{\wti\vf'}\\
S''\ar[r]_{\rho''}&\widetilde{S}.
}
$$

By assumption, the extension 
$R/\fm\to  \wti S/\wti\n$ is 
separable. 
Hence by Fact~\ref{fact120418a}\eqref{fact120418a3},
there exists a comparison $\upsilon\colon S''\to S'$ 
from $R\xra{\beta'\dot{\varphi}} S''\xra{\rho''} \widetilde{S}$
to $R\xra{\dot{\wti\vf}\alpha} S'\xra{\wti\vf'}\widetilde{S}$.
In particular, the following diagram commutes:
$$\xymatrix{
&S''\ar@{-->}[d]^{\upsilon}\ar[rd]^{\rho''}&\\
R\ar[r]^{\dot{\wti\vf}\alpha}\ar[ur]^{\beta'\dot{\varphi}}&S'\ar[r]^{\wti\vf'}&\widetilde{S}.
}$$
Combining this with~\eqref{eq120208a}, we obtain the
next commutative diagram:
\begin{equation}\label{eq120215a}
\begin{split}
\xymatrix{
R\ar[r]^{\dot{\varphi}}\ar[dd]_{\alpha}&R'\ar[d]_{\beta'}\ar[r]^{\varphi'}&S\ar[dd]^{\beta}\\
&S''\ar[rd]^{\rho''}\ar@{-->}[d]_{\upsilon}&\\
\wti R\ar[r]_{\dot{\wti\vf}}&S'\ar[r]_{\wti\vf'}
&\widetilde{S}}
\end{split}\end{equation}

We claim that $\upsilon$ is an isomorphism. Once this is shown, it follows
from the commutativity of~\eqref{eq120215a} that the map
$\alpha'=\upsilon\beta'$ makes the diagram~\eqref{eq120209a} commute.
The fact that the diagram~\eqref{eq120208a} is a pushout then implies that
the second square in the diagram~\eqref{eq120209a} is also a pushout. 
Lastly, the fact that $\beta'$ is weakly regular implies that
$\alpha'$ is also weakly regular, so the desired conclusions hold.

To prove that $\upsilon$ is an isomorphism
it is
enough to show that both Cohen factorizations
$R\xra{\beta'\dot{\varphi}} S''\xra{\rho''} \widetilde{S}$
and $R\xra{\dot{\wti\vf}\alpha} S'\xra{\wti\vf'} \widetilde{S}$
are minimal; see Fact~\ref{fact120418a}\eqref{fact120418a3}. 
Therefore by definition we need to show the following equalities:
$$
\dim(S'')-\dim(R)\stackrel{(*)}{=}\edim(\widetilde{S}/\fm \widetilde{S})\stackrel{(\dagger)}{=}\dim(S')-\dim(R).
$$
We begin with equality $(*)$. 
The fact that $R'\xra{\beta'} S''\xra{\rho''} \widetilde{S}$ is a minimal
Cohen factorization explains the first step in the next sequence:
$$
\dim(S'')-\dim(R')=\edim(\widetilde{S}/\fm'\widetilde{S})=\edim(\widetilde{S}/\fn\wti{S}).
$$
The second equality is from the fact that the map $R'\xra{\vf'} S$ is surjective.
Therefore, since $R\xra{\dot{\varphi}} R'\xra{\varphi'} S$
is a minimal regular factorization, we conclude that
\begin{align*}
\dim(S'')-\dim(R)&=(\dim(S'')-\dim(R'))+(\dim(R')-\dim(R))\\
&= \edim(\widetilde{S}/\fn\wti{S})+\edim(S/\fm S).
\end{align*}
Now,  the induced map $\overline{\beta}\colon S/\fm S\to \widetilde{S}/\fm \widetilde{S}$ is
a flat local homomorphism. 
Furthermore, the closed fiber of this map is the same as the closed fiber for $\beta$,
hence it is regular. 
Using Lemma~\ref{for Cohen fact 1}, we conclude that
$$\edim(\widetilde{S}/\fm \widetilde{S})=\edim(\wti S/\n\wti S)+\edim(S/\fm S).$$
The equality $(*)$ follows from this with the previous display.

Next, we explain the  equality $(\dagger)$.
By assumption $\widetilde{R}\xra{\dot{\wti\vf}} S'\xra{\wti\vf'} \widetilde{S}$ is a minimal Cohen factorization and $\alpha\colon R\to \widetilde{R}$ is a flat local homomorphism. This explains the 
second step in the next display:
\begin{align*}
\dim(S')-\dim(R)&=(\dim(S')-\dim(\widetilde{R}))+(\dim(\widetilde{R})-\dim(R))\\
&=\edim(\widetilde{S}/\widetilde{\fm} \widetilde{S})+\dim(\widetilde{R}/\fm \widetilde{R})\\
&=\edim(\widetilde{S}/\wti\fm \widetilde{S})+\edim(\widetilde{R}/\fm \widetilde{R})\\
&=\edim(\wti S/\m\wti S).
\end{align*}
The third step comes from the fact that $\alpha$ is  weakly regular.
The fourth step is from Fact~\ref{fact120215a}
since the maps $R\xra{\alpha}\wti R\xra{\wti\vf}\wti S$
are weakly regular. This display explains $(\dagger)$, and the proof is complete.
\end{proof}

Our next result complements the previous one
by providing a horizontal map in the proof of
Theorem~\ref{intthm120201a}.

\begin{prop}\label{lem120222a}
Consider a commutative diagram of local ring homomorphisms
\begin{equation}\label{eq120222b}
\begin{split}
\xymatrix{
(R,\fm)\ar[r]^{\varphi}\ar[d]_{\alpha}&(S,\fn)\ar[d]^{\beta}\\
(\widetilde{R},\widetilde{\fm})\ar[r]^{\widetilde{\varphi}}&(\widetilde{S},\widetilde{\fn})
}
\end{split}\end{equation}
such that $\alpha$ has a regular factorization
$R\xra{\dot{\alpha}} R''\xra{\alpha'} \wti R$,
the ring $\widetilde{S}$ is complete, and
the  field extension
$R/\m\to \widetilde{S}/\widetilde{\fn}$  
is separable.
Let
$S\xra{\dot{\beta}} S'\xra{\beta'} \wti S$
be a  Cohen factorization
of $\beta$.
Then there is a commutative diagram of local ring
homomorphisms 
\begin{equation}\label{eq120222c}
\begin{split}
\xymatrix{
R\ar[r]^{\vf}\ar[d]_{\dot\alpha}&S\ar[d]^{\dot\beta}\\
R''\ar[r]^{\sigma}\ar[d]_{\alpha'}&S'\ar[d]^{\beta'}\\
\widetilde{R}\ar[r]^{\dot{\vf}}&\widetilde{S}.
}
\end{split}\end{equation}
\end{prop}

\begin{proof}
We prove the result in two cases.

Case 1: the ring $S$ is complete.
Let 
$R\xra{\dot{\varphi}} R'\xra{\varphi'} S$
and
$\widetilde{R}\xra{\dot{\wti\vf}} \wti R'\xra{\wti\vf'} \widetilde{S}$
be   Cohen factorizations
of $\vf$ and $\widetilde{\varphi}$, respectively. 
Since $\wti R'$ and $S'$ are complete, the maps
$R''\xra{\dot{\wt{\vf}}\alpha'}\wti{R}'$
and $R'\xra{\dot\beta\vf'}S'$ have Cohen factorizations
$R''\xra{\dot\pi}T\xra{\pi'}\wti{R}'$
and
$R'\xra{\dot\tau}T'\xra{\tau'}S'$,
respectively.
Since diagram~\eqref{eq120222b} commutes,
one sees readily that
$\wti\vf'\pi'\dot\pi\dot\alpha=\beta'\tau'\dot\tau\dot\vf$.
In other words, the diagrams
\begin{align}
R\xra{\dot\pi\dot\alpha}T\xra{\wti\vf'\pi'}\wti S
&&
R\xra{\dot\tau\dot\vf}T'\xra{\beta'\tau'}\wti S
\label{eq120222d}
\end{align}
are Cohen factorizations of the same map;
see Fact~\ref{fact120215a}.

By assumption, the induced field extension
$R/\m\to\wti S/\wti\n$ is separable.
Thus, from Fact~\ref{fact120418a}\eqref{fact120418a3} we conclude that there
is a comparison $v\colon T\to T'$
between the factorizations~\eqref{eq120222d}.
It follows that the next diagram commutes
$$\xymatrix{
R\ar[rr]^{\dot{\varphi}}\ar[d]_{\dot\alpha}&&R'\ar[r]^{\varphi'}\ar[d]^{\dot\tau}&S\ar[d]^{\dot\beta}\\
R''\ar[r]^{\dot\pi}\ar[d]_{\alpha'}&T\ar[r]^-v\ar[d]^{\pi'}& T'\ar[r]^{\tau'} &S'\ar[d]^{\beta'}\\
\widetilde{R}\ar[r]^{\dot{\wti\vf}}&\wti{R}'\ar[rr]^{\wti\vf'}&&\widetilde{S}.
}$$
So the map $\sigma=\tau' v\dot\pi\colon R''\to S'$
makes diagram~\eqref{eq120222c} commute.

Case 2: the general case.
Let $\Comp{(-)}n$ denote the $\n$-adic completion functor,
and consider the following  commutative diagram of local ring homomorphisms
$$\xymatrix{
S\ar[r]\ar[d]_{\dot\beta}
&\Comp Sn\ar[d]^{\Comp{\dot\beta}n} \\
S'\ar[r]^{\cong}&\Comp{S'}n}$$
where the horizontal maps are the natural ones.
The map $S'\to\Comp{S'}n$ is an isomorphism since $S'$ is complete
and $\dot\beta$ is local.
Accordingly, we identify $S'$ and $\Comp{S'}n$, and similarly
for $\wti S$ and $\Comp{\wti{S}}n$.
From~\cite[(1.9) Remark]{avramov:solh} we know that 
the diagram $\Comp Sn\xra{\Comp{\dot\beta}n}S'\xra{\beta'}\wti S$
is a Cohen factorization of the map $\Comp Sn\xra{\Comp\beta n}\wti S$.
The previous diagram provides  the second commutative triangle in the next diagram,
and the other triangle commutes by definition of $\grave\vf$:
$$
\xymatrix{
&&S\ar[ld]\ar[ldd]^{\dot\beta} \\
R\ar[r]_{\grave\vf}\ar[rru]^{\vf}\ar[d]_{\dot\alpha}&\Comp Sn\ar[d]_{\Comp{\dot\beta}n}\\
R''\ar@{-->}[r]^{\sigma}\ar[d]_{\alpha'}&S'\ar[d]^{\beta'}\\
\widetilde{R}\ar[r]^{\dot{\vf}}&\widetilde{S}.}
$$
Case 1 provides a local ring homomorphism $\sigma$ making the two squares
commute. It follows readily that $\sigma$ also makes the diagram~\eqref{eq120222c}
commute.
\end{proof}

\begin{para}[Proof of Theorem~\ref{intthm120201a}]
\label{proof120417a}
By Proposition~\ref{lem120222a} there is a local ring homomorphism
$\sigma\colon R''\to S'$ making the next diagram commute:
\begin{equation}\label{eq120223a}
\begin{split}
\xymatrix{
R\ar[r]^{\dot{\varphi}}\ar[d]_{\dot\alpha}&R'\ar[r]^{\varphi'}&S\ar[d]^{\dot\beta}\\
R''\ar[rr]^{\sigma}\ar[d]_{\alpha'}&&S'\ar[d]^{\beta'}\\
\widetilde{R}\ar[r]^{\dot{\wti\vf}}&\wti{R}'\ar[r]^{\wti\vf'}&\widetilde{S}.
}
\end{split}\end{equation}
Since $S'$ is complete, the map $\sigma$ has a Cohen factorization
$R''\xra{\dot\sigma}T'\xra{\sigma'}S'$.
Proposition~\ref{thm120201a} provides a weakly regular local ring homomorphism
$\delta\colon R'\to T'$
making the next diagram commute,
where
$R''\xra{\dot\pi}T''\xra{\pi'}\wti{R}'$ is a Cohen factorization
of $R''\xra{\dot{\wti\vf}\alpha'}\wti{R}'$:
$$\xymatrix{
R\ar[rr]^{\dot{\varphi}}\ar[d]_{\dot\alpha}&&R'\ar[r]^{\varphi'}\ar[d]^{\delta}&S\ar[d]^{\dot\beta}\\
R''\ar[rr]^{\dot\sigma}\ar[dd]_{\alpha'}\ar[rd]^-{\dot\pi}&&T'\ar[r]^{\sigma'}&S'\ar[dd]^{\beta'}\\
&T''\ar[rd]^-{\pi'}\\
\widetilde{R}\ar[rr]^{\dot{\wti\vf}}&&\wti{R}'\ar[r]^{\wti\vf'}&\widetilde{S}.
}$$
Fact~\ref{fact120215a} implies that the diagrams
$R''\xra{\dot\pi}T''\xra{\wti{\vf}'\pi'}\wti S$
and
$R''\xra{\dot\sigma}T'\xra{\beta'\sigma'}\wti S$
are two Cohen factorizations of the map 
$\wti\vf\alpha'=\beta'\sigma$.
By assumption~\eqref{intthm120201a3}, the induced field extension
$R/\m\to\wti S/\wti \n$ is separable, so there is a comparison
$v\colon T'\to T''$ of these Cohen factorizations, 
by Fact~\ref{fact120418a}\eqref{fact120418a3}. This is a local ring homomorphism
making the next diagram commute:
\begin{equation}\label{eq120229a}
\begin{split}
\xymatrix{
R\ar[r]^{\dot{\varphi}}\ar[d]_{\dot\alpha}&R'\ar[r]^{\varphi'}\ar[d]^{\delta}&S\ar[d]^{\dot\beta}\\
R''\ar[r]^{\dot\sigma}\ar[d]_{\alpha'}&T'\ar[r]^{\sigma'}\ar[d]^{\pi'v}&S'\ar[d]^{\beta'}\\
\widetilde{R}\ar[r]^{\dot{\wti\vf}}&\wti{R}'\ar[r]^{\wti\vf'}&\widetilde{S}.
}\end{split}\end{equation}
At this point, the only problem is that $\pi'v$ may not be surjective.
We use a technique from~\cite{avramov:solh, grothendieck:ega4-4} to remedy this.

Since $\wti R'$ and $S'$ are local and complete and 
the maps $\wti\vf'$ and $\beta'$ are surjective, 
in the pull-back square
$$\xymatrix{
\wti T\ar[r]^{f} \ar[d]_{g}\ar@{}[rd]|<<{\ulcorner}& S'\ar[d]^{\beta'} \\
\wti R'\ar[r]^{\wti\vf'}&\wti S
}$$
the ring $\wti T$ is noetherian, local, and complete by~\cite[(19.3.2.1)]{grothendieck:ega4-4}.
The commutativity of the diagram~\eqref{eq120229a} implies that there is a local
ring homomorphism $w\colon T'\to\wti T$ making the next diagram commute:
\begin{equation}\label{eq120229b}
\begin{split}
\xymatrix{
R\ar[r]^{\dot{\varphi}}\ar[d]_{\dot\alpha}&R'\ar[rr]^{\varphi'}\ar[d]^{\delta}&&S\ar[d]^{\dot\beta}\\
R''\ar[r]^{\dot\sigma}\ar[dd]_{\alpha'}&T'\ar[rr]^{\sigma'}\ar[dd]_{\pi'v}\ar[rd]^w&&S'\ar[dd]^{\beta'}\\
&&\wti T\ar[ru]^f\ar[ld]_g\\
\widetilde{R}\ar[r]^{\dot{\wti\vf}}&\wti{R}'\ar[rr]^{\wti\vf'}&&\widetilde{S}.
}\end{split}\end{equation}
The ring $\wti T$ is complete, so the map $w$ has a Cohen factorization
$T'\xra{\dot w}T\xra{w'}\wti T$.
Consider the following diagram
$$\xymatrix{
R\ar[r]^{\dot{\varphi}}\ar[d]_{\dot\alpha}&R'\ar[r]^{\varphi'}\ar[d]^{\dot w\delta}&S\ar[d]^{\dot\beta}\\
R''\ar[r]^{\dot w\dot{\sigma}}\ar[d]_{\alpha'}&T\ar[r]^{fw'}\ar[d]^{gw'}&S'\ar[d]^{\beta'}\\
\widetilde{R}\ar[r]^{\dot{\wti\vf}}&\wti{R}'\ar[r]^{\wti\vf'}&\widetilde{S}.
}$$
This diagram has the  desired properties for~\eqref{eq120222z}.
Indeed, commutativity follows from the commutativity of~\eqref{eq120229b},
using the equation $w=w'\dot w$, and 
the middle column and middle row are Cohen factorizations
by Fact~\ref{fact120215a}.
\qed
\end{para}

The following example, from~\cite[(1.8) Example]{avramov:solh} shows that
the separability assumptions~\eqref{intthm120201a3} 
in Theorem~\ref{intthm120201a} are necessary.

\begin{ex}\label{ex130327a}
In~\cite[(1.8) Example]{avramov:solh}, the authors construct a
local ring homomorphism $\vf\colon R\to S$
with a pair of Cohen factorizations
\begin{equation}
\label{eq130327a}
\begin{split}
\xymatrix{
& R_1\ar[rd]^{\vf_1'}\ar@{..>}[dd]^<<<<<<{\not\exists} \\
R\ar[ur]^{\dot\vf_1}\ar[rr]^<<<<<<<{\vf}\ar[dr]_{\dot\vf_2}&&S \\
& R_2\ar[ru]_{\vf_2'}}
\end{split}\end{equation}
such that there is not a comparison from the top one to the bottom one.
(Note that, by Fact~\ref{fact120418a}\eqref{fact120418a3},
the map on  residue fields induced by $\vf$ cannot be separable.)
We begin with the following diagram, as in~\eqref{eq120222a}
\begin{equation*}
\begin{split}
\xymatrix{
R\ar[r]^{\varphi}\ar[d]_{\id}&S\ar[d]^{\id}\\
R\ar[r]^{\varphi}&S
}
\end{split}\end{equation*}
with the given Cohen factorizations represented by 
the horizontal solid arrows in the next diagram.
\begin{equation*} 
\begin{split}
\xymatrix{
R\ar[r]^{\dot{\varphi}_1}\ar[d]_{\id}&R_1\ar[r]^{\varphi'_1}\ar@{..>}[d]^{\dot\gamma}&S\ar[d]^{\id}\\
R\ar@{..>}[r]^{\dot{\sigma}}\ar[d]_{\id}&T\ar@{..>}[r]^{\sigma'}\ar@{..>}[d]^{\gamma'}&S\ar[d]^{\id}\\
R\ar[r]^{\dot{\vf}_2}&R_2\ar[r]^{\vf_2'}&S
}
\end{split}\end{equation*}
If there were Cohen factorizations (as represented
by the dotted arrows in the diagram) 
then the map $\gamma'\dot\gamma$ would provide
a comparison of the Cohen factorizations from~\eqref{eq130327a},
contradicting the conclusion of~\cite[(1.8) Example]{avramov:solh}.
Thus, in the notation of Theorem~\ref{intthm120201a},
the field extension $\wti R/\wti\m\to\wti S/\wti\n$ must be
separable (even when the extension $R/\m\to \wti R/\wti\m$ is trivial).
Similarly, the reflected diagrams
\begin{equation*}
\begin{split}
\xymatrix{
R\ar[r]^{\id}\ar[d]_{\varphi}&R\ar[d]^{\varphi}\\
S\ar[r]^{\id}&S
}
\qquad\qquad
\xymatrix{
R\ar[d]_{\dot{\varphi}_1}\ar[r]^{\id}&R\ar[r]^{\id}\ar@{..>}[d]^{\dot\gamma}&R\ar[d]^{\dot{\vf}_2}\\
R_1\ar@{..>}[r]^{\dot{\sigma}}\ar[d]_{\varphi'_1}
&T\ar@{..>}[r]^{\sigma'}\ar@{..>}[d]^{\gamma'}&R_2\ar[d]^{\vf_2'}\\
S\ar[r]^{\id}&S\ar[r]^{\id}&S
}
\end{split}\end{equation*}
show that the field extension $R/\m\to \wti R/\wti\m$ must be
separable (even when the extension $\wti R/\wti\m\to\wti S/\wti\n$  is trivial).
\end{ex}

The next example shows that, even when the factorizations
$R\xra{\dot{\varphi}} R'\xra{\varphi'} S$
and 
$\widetilde{R}\xra{\dot{\wti\vf}} \wti R'\xra{\wti\vf'} \widetilde{S}$
are minimal, the map $\gamma=\gamma'\dot\gamma$ 
in Theorem~\ref{intthm120201a} is not uniquely
determined.

\begin{ex}\label{ex120306a}
Let $k$ be a field, and consider the following diagram where the unspecified maps are the
natural ones:
$$\xymatrix{
k \ar[r]\ar[d]&\power kX\ar[r]\ar@{-->}[d]^{\gamma}&\power kX\ar[d]\\
k\ar[r]&\power kX\ar[r]&\power kX/(X^2).
}
$$
One can define $\gamma$ by mapping $X$ to any power series of the form
$X+X^2f$ to make the diagram commute.
\end{ex}

\section{Properties of Local Ring Homomorphisms in Commuting Diagrams}\label{sec120306a}

In this section, we prove Theorem~\ref{prop120322a} from the introduction.

\begin{lem}\label{intthm120201b}
Fix a commutative diagram of  field extensions
\begin{equation}\label{eq120503a}
\begin{split}
\xymatrix{k\ar[r]^{\vf_0}\ar[d]_{\alpha_0}
&l\ar[d]^{\beta_0}\\
\wti k\ar[r]^{\wti\vf_0}&\wti l}
\end{split}\end{equation}
Let $\varphi\colon(R,\fm,k)\to (S,\fn,l)$ 
and  $\alpha\colon R\to (\widetilde{R},\widetilde{\fm},\wti k)$
and $\beta\colon S\to (\widetilde{S},\widetilde{\fn},\wti l)$
be local ring homomorphisms
that induce the maps $\vf_0$ and $\alpha_0$ and $\beta_0$ on residue fields.
If $\alpha$ is weakly Cohen and $\wti S$ is complete,
then there is a commutative diagram of local ring homomorphisms
$$
\xymatrix{
(R,\fm,k)\ar[r]^{\varphi}\ar[d]_{\alpha}&(S,\fn,l)\ar[d]^{\beta}\\
(\widetilde{R},\widetilde{\fm},\wti k)\ar@{-->}[r]^{\widetilde{\varphi}}&(\widetilde{S},\widetilde{\fn},\wti l)
}
$$
such that $\wti\vf$ induces the map $\wti\vf_0$ on residue fields.
\end{lem}

\begin{proof}
By assumption, the maps $\alpha$ and $\beta$ make the following diagram commute
$$
\xymatrix{
R\ar[rrr]^-{\vf}\ar[rd]\ar[ddd]_-{\alpha}
&&& S\ar[ddd]^-{\beta}\ar[ld] \\
&k\ar[r]\ar[d]&l\ar[d] \\
&\wti k\ar[r]&\wti l \\
\wti R\ar[ur]
&&&\wti S\ar[ul]
}
$$
Since the extension $k\hookrightarrow \wti k$ is separable, the map $\alpha$
is formally smooth by~\cite[Th\'eor\`eme (19.8.2.i)]{grothendieck:ega4-1},
so~\cite[Corollaire (19.3.11)]{grothendieck:ega4-1} provides a local homomorphism 
$\widetilde{\varphi}\colon \widetilde{R}\to \widetilde{S}$ 
making the next  diagram commute:
$$
\xymatrix{
R\ar[rrr]^-{\vf}\ar[rd]\ar[ddd]_-{\alpha}
&&& S\ar[ddd]^-{\beta}\ar[ld] \\
&k\ar[r]\ar[d]&l\ar[d] \\
&\wti k\ar[r]&\wti l \\
\wti R\ar[ur]\ar@{-->}[rrr]^-{\wti\vf}
&&&\wti S\ar[ul]
}
$$
It follows that the outer square in this diagram has the desired properties.
\end{proof}

\begin{disc}
Given a diagram~\eqref{eq120503a},
there exist  Cohen maps $\alpha$ and $\beta$,
as in Lemma~\ref{intthm120201b},
by~\cite[Th\'{e}or\`{e}m 19.8.2(ii)]{grothendieck:ega4-1}.
\end{disc}

We use the next
result to relate the flat dimensions 
of the maps in Theorem~\ref{prop120322a}.

\begin{prop}
\label{prop120312b}
Consider a commutative diagram of local ring homomorphisms
\begin{equation}\label{eq120503b}
\begin{split}
\xymatrix{
(R,\fm,k)\ar[r]^{\varphi}\ar[d]_{\alpha}&(S,\fn,l)\ar[d]^{\beta}\\
(\widetilde{R},\widetilde{\fm},\wti k)\ar[r]^{\widetilde{\varphi}}&(\widetilde{S},\widetilde{\fn}, \wti l)
}
\end{split}\end{equation}
such that $\fd(\alpha)$ and $\fd(\beta)$ are both finite.
\begin{enumerate}[\rm(a)]
\item 
\label{prop120312b1}
One has $\fd(\vf)\leq\fd(\wti\vf)+\fd(\alpha)$; in particular,
if $\fd(\wti\vf)$ is finite, then so is $\fd(\vf)$.
\item 
\label{prop120312b2}
If $\alpha$ is weakly regular, then $\fd(\wti\vf)\leq\edim(\alpha)+\fd(\beta)+\fd(\vf)$;
in particular, $\fd(\vf)$ and
$\fd(\wti\vf)$ are simultaneously finite.
\item 
\label{prop120312b3}
If $\alpha$ is flat with closed fiber  a field and $\beta$ is flat,
then  $\fd(\widetilde\varphi)=\fd({\varphi})$.
\end{enumerate}
\end{prop}

\begin{proof}
\eqref{prop120312b1}
This is proved like~\cite[Theorem 5.7]{iyengar:golh}.

\eqref{prop120312b2}
Assume without loss of generality that $\fd(\vf)<\infty$.
Let $\x=x_1,\ldots,x_e\in\wti\m$ be a sequence whose residue sequence
in $\wti R/\m\wti R$ is a regular system of parameters.
This explains the first isomorphism in the derived category $\cat{D}(\wti S)$
in the next sequence:
\begin{align*}
\Lotimes[\wti R]{\wti k}{\wti S}
&\simeq\Lotimes[\wti R]{[\Lotimes[\wti R]{(\wti R/\x\wti R)}{(\wti R/\m\wti R)}]}{\wti S}\\
&\simeq\Lotimes[\wti R]{(\wti R/\x\wti R)}{[\Lotimes[\wti R]{(\wti R/\m\wti R)}{\wti S}]}\\
&\simeq\Lotimes[\wti R]{(\wti R/\x\wti R)}{[\Lotimes[\wti R]{(\Lotimes{k}{\wti R})}{\wti S}]}\\
&\simeq\Lotimes[\wti R]{(\wti R/\x\wti R)}{[\Lotimes[S]{(\Lotimes{k}{S})}{\wti S}]}.
\end{align*}
The second isomorphism is associativity.
The third isomorphism is from the flatness of $\alpha$.
The fourth isomorphism is from the commutativity of~\eqref{eq120503b}.
This explains the second step in the next display:
\begin{align*}
\fd(\wti\vf)
&=\amp(\Lotimes[\wti R]{\wti k}{\wti S})\\
&=\amp(\Lotimes[\wti R]{(\wti R/\x\wti R)}{[\Lotimes[S]{(\Lotimes{k}{S})}{\wti S}]})\\
&\leq\pd_{\wti R}(\wti R/\x\wti R)+\amp(\Lotimes[S]{(\Lotimes{k}{S})}{\wti S})\\
&=e+\amp(\Lotimes[S]{(\Lotimes{k}{S})}{\wti S})\\
&= e+\fd(\beta\vf)\\
&\leq e+\fd(\beta)+\fd(\vf)\\
&=\edim(\wti R/\m\wti R)+\fd(\beta)+\fd(\vf).
\end{align*}
The first and fifth steps are from~\cite[Proposition 5.5]{avramov:hdouc}.
The third step is standard; see~\cite[(7.28) and (8.17)]{foxby:hacr}.
The fourth step is from~\cite[Corollary to Theorem 22.5]{matsumura:crt},
which tells us that $\x$ is $\wti R$-regular.
The sixth step is by~\cite[(1.8(a))]{avramov:lgh}.
The seventh step is by definition of $e$ as $\edim(\wti R/\m\wti R)$.

\eqref{prop120312b3}
Part~\eqref{prop120312b1}
implies that $\fd(\widetilde\varphi)\geq\fd({\varphi})$,
and~\eqref{prop120312b2}
implies that $\fd(\widetilde\varphi)\leq\fd({\varphi})$.
\end{proof}

Here are some examples that show the limitations of the previous result.

\begin{ex}\label{ex120313a}
Let $k$ be a field, and consider the following commutative diagram of 
natural local ring homomorphisms:
$$\xymatrix{
k\ar[r]^-{\vf}\ar[d]_-{\alpha}
& \power kX\ar[d]^-{\beta} \\
\power kX\ar[r]^-{\wti\vf}
&\power kX/(X^2).
}$$
Note that $\alpha$ is weakly regular, and we have
$$\fd(\wti\vf)=1<2=1+1+0=\edim(\alpha)+\fd(\beta)+\fd(\vf)$$
so we can have strict inequality in Proposition~\ref{prop120312b}\eqref{prop120312b2}.
\end{ex}

\begin{ex}\label{ex120313b}
Let $k$ be a field, and let $S$ be a  complete local ring with
coefficient field $k$, and assume that $S$ is not a field.
Consider the following commutative diagram of 
natural local ring homomorphisms:
$$\xymatrix{
k\ar[r]^-{\vf}\ar[d]_-{\alpha}
& S\ar[d]^-{\beta} \\
k\ar[r]^-{\wti\vf}
&k.
}$$
Then we have
$\fd(\widetilde\varphi)=0<0+\fd(\beta)=\fd({\varphi})+\fd(\beta)$.
so we can have strict inequality in Proposition~\ref{prop120312b}\eqref{prop120312b3}
if $\beta$ is not flat.
\end{ex}

\begin{lem}\label{prop120312d}
Let $(\wti R',\wti\m',\wti k')\xla{\alpha'}(R',\m',k')\xra{\rho} (R_1,\m_1,k_1)\xla\tau (Q,\fr,k_1)$ be local ring homomorphisms such that $\alpha'$ is  Cohen, 
$\rho$ is flat, and
$\tau$ is surjective.
Fix a commutative diagram 
$$\xymatrix{k'\ar[r]^-{\rho_0}\ar[d]_-{\alpha'_0}
&k_1\ar[d]^-{\alpha_0}\\
\wti k'\ar[r]^-{\wti\rho_0}&\wti k_1}$$
of  field extensions such that $\rho_0$ and $\alpha'_0$ are the maps induced
on residue fields by $\rho$ and $\alpha'$.
Fix a weakly regular  homomorphism
$\alpha_1\colon R_1\to (\widetilde{R}_1,\widetilde{\m}_1,\wti k_1)$
that induces the map $\alpha_0$  on residue fields.
If  $\wti R_1$ is complete,
then there is a commutative diagram of local ring homomorphisms
$$
\xymatrix{
R'\ar[r]^{\rho}\ar[d]_-{\alpha'}&R_1\ar[d]^-{\alpha_1}
&Q\ar[l]_-\tau\ar[d]^-{\delta}\\
\widetilde{R}'\ar[r]^{\widetilde{\rho}}&\widetilde{R}_1
&\wti Q\ar[l]_-{\wti\tau}
}
$$
such that $\wti\rho$ is flat, $\delta$ is weakly regular, the right-hand square is a pushout, and $\wti\rho$
induces the map $\wti\rho_0$ on residue fields.
\end{lem}

\begin{proof}
Lemma~\ref{intthm120201b} provides a local ring homomorphism
$\wti\rho\colon\wti R'\to\wti R_1$ making the left-hand square commute.
Proposition~\ref{prop120312b}\eqref{prop120312b3} implies that $\wti\rho$ is flat.
We get the right-hand square from
Fact~\ref{fact120418a}\eqref{fact120418a2}. 
\end{proof}

The next result augments~\cite[(1.13) Proposition]{avramov:cid}.

\begin{prop} \label{prop120312e}
Let $\alpha'\colon(R',\m',k')\to (\wti R',\wti \m',\wti k')$ be a  Cohen homomorphism, 
and let $M$ be a homologically finite
$R'$-complex. 
Then we have
\begin{enumerate}[\rm(a)]
\item \label{prop120312e1}
$\cidim_{\wti R'}(\Lotimes[R']{\wti R'}{M})=\cidim_{R'}(M)$;
thus, the quantities 
$\cidim_{\wti R'}(\Lotimes[R']{\wti R'}{M})$ and $\cidim_{R'}(M)$
are simultaneously finite.
\item \label{prop120312e2}
$\cmdim_{\wti R'}(\Lotimes[R']{\wti R'}{M})=\cmdim_{R'}(M)$;
in particular, 
$\cmdim_{\wti R'}(\Lotimes[R']{\wti R'}{M})$ and $\cmdim_{R'}(M)$
are simultaneously finite.
\end{enumerate}
\end{prop}

\begin{proof}
\eqref{prop120312e1}
It suffices by Fact~\ref{fact120312a}
to assume that $\cidim_{R'}(M)$ is finite and prove that
$\cidim_{\wti R'}(\Lotimes[\wti R']{\wti R'}{M})$ is finite as well.
Let 
$R'\xra{\rho} (R_1,\m_1,k_1)\xla\tau (Q,\fr,k_1)$
be a quasi-deformation such that
$\pd_Q(\Lotimes[R']{R_1}{M})<\infty$.
Let $\wti k_1$ be a join of $k_1$ and $\wti k'$ over 
 $k'$, so there is a commutative diagram 
$$\xymatrix{k'\ar[r]\ar[d]
&k_1\ar[d]\\
\wti k'\ar[r]&\wti k_1}$$
of  field extensions.
There is a weakly regular  homomorphism
$\alpha_1\colon R_1\to (\widetilde{R}_1,\widetilde{\m}_1,\wti k_1)$
that induces the map $k_1\to \wti k_1$  on residue fields
by~\cite[Proposition (0.10.3.1)]{grothendieck:ega3-1}.
Compose with the natural map from $\wti R_1$ to its completion if necessary
to assume that $\wti R_1$ is complete.
Thus, Lemma~\ref{prop120312d}
provides a commutative diagram of local ring homomorphisms
$$
\xymatrix{
R'\ar[r]^{\rho}\ar[d]_-{\alpha'}&R_1\ar[d]^-{\alpha_1}
&Q\ar[l]_-\tau\ar[d]^-{\delta}\\
\widetilde{R}'\ar[r]^{\widetilde{\rho}}&\widetilde{R}_1
&\wti Q\ar[l]_-{\wti\tau}
}
$$
such that $\wti\rho$ is flat, $\delta$ is weakly regular, and the right-hand square is a pushout.
In particular, the bottom row of this diagram is a quasi-deformation.
Also, in the following sequence, the second and fourth equalities are from the
flatness of $\delta$:
\begin{align*}
\pd_{\wti Q}(\Lotimes[\wti R']{\wti R_1}{(\Lotimes[R']{\wti R'}{M}}))
&=\pd_{\wti Q}(\Lotimes[R']{\wti R_1}{M}) \\
&=\pd_{\wti Q}(\Lotimes[R']{(\Lotimes[Q]{\wti Q}{R_1})}{M})\\
&=\pd_{\wti Q}(\Lotimes[Q]{\wti Q}{(\Lotimes[R']{R_1}{M})})\\
&=\pd_{Q}(\Lotimes[R']{R_1}{M}).
\end{align*}
By definition, it follows that
$\cidim_{\wti R'}(\Lotimes[\wti R']{\wti R'}{M})$ is finite, as desired.

\eqref{prop120312e2}
This is proved as above, starting with a G-quasi-deformation
$(R',\m',k')\xra{\rho} (R_1,\m_1,k_1)\xla\tau (Q,\fr,k_1)$,
and observing that the diagram
$\wti R'\xra{\wti \rho} \wti R_1\xla{\wti \tau} \wti Q$
is also a a G-quasi-deformation.
This follows from the fact that the flat base change
of a G-perfect ideal is G-perfect, which is readily checked.
\end{proof}

The next result compares to Proposition~\ref{prop120312b}.
It is important for our proof of Theorem~\ref{prop120322a} from the introduction.

\begin{prop}
\label{prop120312c}
Consider a commutative diagram of local ring homomorphisms
$$
\xymatrix{
(R,\fm)\ar[r]^{\varphi}\ar[d]_{\alpha}&(S,\fn)\ar[d]^{\beta}\\
(\widetilde{R},\widetilde{\fm})\ar[r]^{\widetilde{\varphi}}&(\widetilde{S},\widetilde{\fn})
}
$$
such that $\alpha$ is  weakly Cohen,
$\beta$ is  weakly regular, and
the induced map $\widetilde{R}/\widetilde{\fm}\to \widetilde{S}/\widetilde{\fn}$ is separable.
Then we have 
\begin{enumerate}[\rm(a)]
\item \label{prop120312c1}
$\gdim(\widetilde{\varphi})=\gdim(\varphi)+\edim(\alpha)-\edim(\beta)$.
Hence, the quantities
$\gdim(\varphi)$ and $\gdim(\widetilde{\varphi})$ are simultaneously finite.
\item \label{prop120312c2}
If $\beta$ is  Cohen,
then $\cidim(\widetilde{\varphi})=\cidim(\varphi)+\edim(\alpha)$,
so the quantities
$\cidim(\varphi)$ and $\cidim(\widetilde{\varphi})$ are simultaneously finite.
\item \label{prop120312c3}
If $\beta$ is  Cohen,
then $\cmdim(\widetilde{\varphi})=\cmdim(\varphi)+\edim(\alpha)$,
so the quantities
$\cmdim(\varphi)$ and $\cmdim(\widetilde{\varphi})$ are simultaneously finite.
\end{enumerate}
\end{prop}

\begin{proof}
\eqref{prop120312c2}
Assume that $\beta$ is Cohen.

Case 1: $S$ and $\wti S$ are complete.
Let
$R\xra{\dot{\varphi}} R'\xra{\varphi'} S$
be a minimal Cohen factorization of $\vf$,
and let
$\widetilde{R}\xra{\dot{\wti\vf}} S'\xra{\wti\vf'} \widetilde{S}$
be a minimal Cohen factorization
of $\widetilde{\varphi}$. 
Proposition~\ref{thm120201a} provides a weakly regular local ring
homomorphism $\alpha'\colon R'\to S'$ such that the
next diagram commutes
\begin{equation}\label{eq120312a}
\begin{split}
\xymatrix{
R\ar[r]^{\dot{\varphi}}\ar[d]_{\alpha}&R'\ar[r]^{\varphi'}\ar[d]^{\alpha'}&S\ar[d]^{\beta}\\
\widetilde{R}\ar[r]^{\dot{\wti\vf}}&S'\ar[r]^{\wti\vf'}&\widetilde{S}
}
\end{split}\end{equation}
and such that the second square is a pushout.
The residue field extension induced by $\alpha'$ is the same as the one induced by
$\beta$, since the maps $\vf'$ and $\wti\vf'$ are surjective.
Thus, the fact that $\alpha'$ is weakly regular and $\beta$ is Cohen
implies that $\alpha'$ is weakly Cohen.
Furthermore,  $\alpha'$ and $\beta$ have isomorphic closed fibers,
so $\alpha'$ is Cohen.

The pushout condition and the flatness of $\alpha'$ explain the first equality in the next display
$$
\cidim_{S'}(\wti S)=
\cidim_{S'}(\Lotimes[R']{S'}S)=
\cidim_{R'}(S)$$
and the second equality is from  Proposition~\ref{prop120312e}\eqref{prop120312e1}.
With Fact~\ref{fact120215a}, this explains the second equality in the next display:
\begin{align*}
\cidim(\widetilde{\varphi})
&=\cidim_{S'}(\wti S)-\edim(\dot{\wti\vf})\\
&=\cidim_{R'}(S)-[\edim(\dot\vf)+\edim(\alpha')-\edim(\alpha)]\\
&=[\cidim_{R'}(S)-\edim(\dot\vf)]-\edim(\beta)+\edim(\alpha)\\
&=\cidim(\varphi)+\edim(\alpha).
\end{align*}
The first and last equalities are by definition, since $\beta$ is Cohen; 
the third equality follows from
the fact that
$\alpha'$ and $\beta$ have isomorphic closed fibers.

Case 2: the general case.
There is a natural ring homomorphism
$\comp \beta\colon \comp S\to\comp{\wti S}$ making the following diagram commute
$$\xymatrix{
S\ar[r]\ar[d]_-{\beta}
&\comp S\ar[d]^-{\comp\beta}\\
\wti S\ar[r]&\comp{\wti S}}$$
where the horizontal maps are the natural ones.
Proposition~\ref{prop120312b}\eqref{prop120312b3}
implies that $\comp \beta$ is flat, and 
it is straightforward to show that its closed fiber is
isomorphic to the
completion of the closed fiber of $\beta$.
Thus, $\comp \beta$ is Cohen.
With~\cite[2.14.2]{sather:cidfc} this explains the first and third equalities in the 
next display:
\begin{align*}
\cidim(\wti\vf)
=\cidim(\grave{\wti\vf})
=\cidim(\grave\vf)+\edim(\alpha)
=\cidim(\vf)+\edim(\alpha).
\end{align*}
The second equality is from Case 1
applied to the diagram
$$
\xymatrix{
R\ar[r]^{\grave\varphi}\ar[d]_{\alpha}&\comp S\ar[d]^{\comp\beta}\\
\widetilde{R}\ar[r]^{\grave{\widetilde{\varphi}}}&\comp{\widetilde{S}}.
}
$$
This completes the proof of part~\eqref{prop120312c2}.

For parts~\eqref{prop120312c1} and~\eqref{prop120312c3},
argue as above, using~\cite[(4.1.4)]{avramov:rhafgd} 
and Proposition~\ref{prop120312e}\eqref{prop120312e2} in place 
of Proposition~\ref{prop120312e}\eqref{prop120312e1}.
\end{proof}

\begin{para}[Proof of Theorem~\ref{prop120322a}]
\label{proof120417c}
Case 1: P $=$ Gorenstein.
Since $\alpha$ and $\beta$ are both flat with Gorenstein closed fibers, they
are both Gorenstein. It follows that there are integers $a$ and $b$ such that
$I^{\wti R}_{\wti R}(t)=t^aI^R_R(t)$
and
$I^{\wti S}_{\wti S}(t)=t^bI^S_S(t)$.
Thus, there is an integer $c$ such that
$I^{S}_{S}(t)=t^cI^R_R(t)$
if and only if 
there is an integer $d$ such that
$I^{\wti S}_{\wti S}(t)=t^dI^{\wti R}_{\wti R}(t)$.
Proposition~\ref{prop120312b}\eqref{prop120312b2}
says that $\fd(\wti\vf)$ is finite if and only if $\fd(\vf)$ is finite,
so $\vf$ is Gorenstein if and only if $\wti\vf$ is Gorenstein, by definition.

Case 2: P $=$ quasi-Gorenstein.
This follows like Case 1, using 
Proposition~\ref{prop120312c}\eqref{prop120312c1}.

For the remaining cases, 
let
$R\xra{\dot{\varphi}} R'\xra{\varphi'} S$
be a minimal Cohen factorization of $\grave\vf$,
and let
$\widetilde{R}\xra{\dot{\wti\vf}} S'\xra{\wti\vf'} \widetilde{S}$
be a minimal Cohen factorization
of $\grave{\widetilde{\varphi}}$. 
Proposition~\ref{thm120201a} provides a weakly regular local ring
homomorphism $\alpha'\colon R'\to S'$ such that the
next diagram commutes
\begin{equation}\label{eq120322a}
\begin{split}
\xymatrix{
R\ar[r]^{\dot{\varphi}}\ar[d]_{\alpha}&R'\ar[r]^{\varphi'}\ar[d]^{\alpha'}&\comp S\ar[d]^{\beta}\\
\widetilde{R}\ar[r]^{\dot{\wti\vf}}&\wti R'\ar[r]^{\wti\vf'}&\comp{\widetilde{S}}
}
\end{split}\end{equation}
and such that the second square is a pushout.

Case 3: P $=$ complete intersection.
Since the second square of diagram~\eqref{eq120322a} is a pushout,
we have $\ker(\wti\vf')=\ker(\vf')\wti R'$.
The fact that $\alpha'$ is flat and local implies that 
a minimal generating sequence for $\ker(\vf')$ extends to a minimal
generating sequence for $\ker(\vf')\wti R'$, and that this sequence is 
$R'$-regular if and only if it is $\wti R'$-regular.
That is, the ideal $\ker(\wti\vf')=\ker(\vf')\wti R'$ in $\wti R'$ is a complete intersection
if and only if $\ker(\vf')\subseteq R'$ is a complete intersection.
Thus,
$\vf$ is complete intersection if and only if $\wti\vf$ is complete intersection, 
by definition.

Case 4: P $=$ (quasi-)Cohen-Macaulay.
Assume for the moment that $\gdim(\vf)<\infty$.
Then
Proposition~\ref{prop120312c}\eqref{prop120312c1}
implies that $\gdim(\wti\vf)<\infty$.
From~\cite[(6.7) Lemma]{avramov:rhafgd},
a relative dualizing complex for $\grave\vf$ is
$D^{\grave\vf}=\Rhom[R']{\comp S}{R'}$,
and it follows that
\begin{align*}
D^{\grave{\wti\vf}}
&=\Rhom[\wti R']{\comp{\wti S}}{\wti R'}\\
&\simeq\Rhom[\wti R']{\Lotimes[R']{\wti R'}{\comp S}}{\Lotimes[R']{\wti R'}{R'}}\\
&\simeq\Lotimes[R']{\wti R'}{\Rhom[R']{\comp S}{R'}}\\
&=\Lotimes[R']{\wti R'}{D^{\grave\vf}}.
\end{align*}
The second step is from the fact that
the second square of diagram~\eqref{eq120322a} is a pushout
and $\alpha'$ is flat.
It follows that
\begin{align*}
\qcmd(\wti\vf)
&=\qcmd(\grave{\wti\vf})
=\amp(D^{\grave{\wti\vf}})
=\amp(\Lotimes[R']{\wti R'}{D^{\grave\vf}})\\
&=\amp(D^{\grave\vf})
=\qcmd(\grave\vf)
=\qcmd(\vf).
\end{align*}
A similar argument shows that
if $\fd(\vf)<\infty$, then $\cmd(\wti\vf)=\cmd(\vf)$.

In view of the simultaneous finiteness given in
Propositions~\ref{prop120312b}\eqref{prop120312b1}--\eqref{prop120312b2}
and~\ref{prop120312c}\eqref{prop120312c1}, 
we see that
$\vf$ is quasi-Cohen-Macaulay if and only if $\wti\vf$ is quasi-Cohen-Macaulay,
and $\vf$ is Cohen-Macaulay if and only if $\wti\vf$ is Cohen-Macaulay.
\qed
\end{para}

\begin{disc}
We have learned from Javier Majadas and Tirdad Sharif that the assumptions of weak Cohenness and separability can be removed from Theorem~\ref{prop120322a} 
in the case where P is ``complete intersection'', using Andr\'e-Quillen homology. In the other cases, though, we do not know if these assumptions are necessary.
\end{disc}

\section{Proof of Theorem~\ref{intthm120201d}}
\label{sec120417a}

\begin{disc}\label{disc120426a}
Complete intersection dimension is quite nice in many respects. 
However, at this time we do not know how it behaves with respect to short exact
sequences: If two of the modules in an exact sequence $0\to M_1\to M_2\to M_3\to 0$
have finite CI-dimension, must the third module also have finite CI-dimension?

The difficulty with this question is the following. Assume, for instance, that
$\cidim_R(M_1)$ and $\cidim_R(M_2)$ are finite. Then for $i=1,2$ there is a 
quasi-deformation
$R\to R_i\from Q_i$  such that
$\pd_{Q_i}(\Otimes{R_i}{M_i})<\infty$.
If there were a single quasi-deformation $R\to R'\from Q$ such that
$\pd_{Q}(\Otimes{R'}{M_i})<\infty$ for $i=1,2$ then we could conclude
easily that $\pd_{Q}(\Otimes{R'}{M_3})<\infty$, so $\cidim_R(M_3)<\infty$.
The difficulty lies in attempting  to combine the two given quasi-deformations
into a single one that works for both $M_1$ and $M_2$. 
Theorem~\ref{intthm120201d} deals with half of this problem  when the maps $R\to R_i$ are weakly Cohen
by showing how to combine the flat maps in the given quasi-deformations.
However, we do not see how to use Theorem~\ref{intthm120201d} to answer the above question in any special cases,
e.g., when $R$ contains a field of characteristic 0.
\end{disc}

\begin{para}[Proof of Theorem~\ref{intthm120201d}]
\label{proof120417b}
\eqref{intthm120201d01}
Consider the induced maps $k\to k_i$ for $i=1, 2$. 
Since these are separable by assumption, there is a commutative diagram
of separable field extensions
$$\xymatrix{
k\ar[r]\ar[d]
&k_1\ar[d] \\
k_2\ar[r]&k'.
}$$
Let $\alpha\colon (R,\m,k)\to (\wti R,\wti \m,k')$ and 
$\beta_i\colon (R_i,\m_i,k_i)\to(\wti R_i,\wti \m_i,k')$ be 
Cohen extensions corresponding to the separable field extensions
$k\to k'$ and $k_i\to k'$ such that the rings $\wti R$ and $\wti R_i$ are
complete; see~\cite[Th\'{e}or\`{e}m 19.8.2(ii)]{grothendieck:ega4-1}.

By Lemma~\ref{intthm120201b} and Proposition~\ref{prop120312b}\eqref{prop120312b3}, each $\varphi_i$ ($i=1, 2$) can be extended to a
flat local homomorphism $\widetilde{\varphi}_i$ making the following diagram commute: 
$$
\xymatrix{
(R,\m,k)\ar[r]^-{\vf_i}\ar[d]_-{\alpha}
&(R_i,\m_i,k_i)\ar[d]^-{\beta_i} \\
(\wti R,\wti\m,k')\ar@{-->}[r]^-{\wti\vf_i}&(\wti R_i,\wti\m_i,k').}
$$
For $i=1,2$, Fact~\ref{fact120418a}\eqref{fact120418a2}
provides
a commutative diagram of local ring homomorphisms
\begin{equation}\label{eq120409a}
\begin{split}
\xymatrix{
(Q_i,\n_i,k_i)\ar@{->>}[d]_-{\tau_i}\ar@{-->}[r]^{\delta_i}
&(\widetilde{Q}_i,\wti\n_i,k')\ar@{-->>}[d]^{\gamma_i}\\
(R_i,\m_i,k_i)\ar[r]^-{\beta_i}&(\widetilde{R}_i,\wti\m_i,k')
}
\end{split}\end{equation}
such that $\widetilde{Q}_i$ is complete, $\delta_i$ is weakly regular, $\gamma_i$ is surjective, and the induced map $R_i\otimes_{Q_i}\widetilde{Q}_i\to \widetilde{R}_i$ is an isomorphism. 
Because of this isomorphism and the flatness of $\delta_i$,
the fact that $\tau_i$ is surjective with kernel generated by a $Q_i$-regular
sequence implies that $\gamma_i$ is surjective with kernel generated by a 
$\wti Q_i$-regular sequence.
Also, since the maps $R\xra{\vf_i}R_i\xra{\beta_i}\wti R_i$ are flat and local,
the same is true of the composition.
Thus, for  $i=1, 2$, the diagram 
$R\xra{\beta_i\vf_i} \widetilde{R}_i\xla{\gamma_i} \widetilde{Q}_i$ is a quasi-deformation.

Now consider the complete tensor product 
$R'=\widetilde{R}_1\widehat{\otimes}_{\widetilde{R}}\widetilde{R}_2$;
see, e.g., \cite[Section 0.7.7]{grothendieck:ega1} and~\cite[Section V.B.2]{serre:alm}
for background on this. This  is a complete semi-local noetherian ring that is flat over 
each $\widetilde{R}_i$ by~\cite[Lemme 19.7.1.2]{grothendieck:ega4-1}. 
Moreover, the proof of~\cite[Lemme 19.7.1.2]{grothendieck:ega4-1} shows that the maximal ideals of $R'$
are in bijection with the maximal ideals of 
$(\wti R_1/\wti\m_1)\otimes_{\wti R}(\wti R_2/\wti\m_2)\cong k'\otimes_{k'}k'
\cong k'$, so $R'$ is local.
For $i=1,2$ let
$\sigma_i\colon(\wti R_i,\wti\m_i,k')\to (R',\m',k')$
be the natural (flat local) map.

Hence, each $\sigma_i$ is flat with complete target,
so  by
Fact~\ref{fact120418a}\eqref{fact120418a}
there exists a commutative diagram of local homomorphisms
\begin{equation}\label{eq120409b}
\begin{split}
\xymatrix{
(\widetilde{Q}_i,\wti\n_i,k')\ar@{->>}[d]_-{\gamma_i}\ar@{-->}[r]^{\overline{\delta}_i}&
(\overline{Q}_i,\ol\n_i,k')\ar@{-->>}[d]^{\overline{\gamma}_i}\\
(\widetilde{R}_i,\wti\m_i,k')\ar[r]^-{\sigma_i}&(R',\m',k')
}
\end{split}\end{equation}
such that $\overline{Q}_i$ is complete, $\overline{\delta}_i$ is weakly regular, 
and $\overline{\gamma}_i$ is surjective. 

Thus, the following diagram
$$
\xymatrix@C=7mm{
(Q_1,\n_1,k_1)\ar[r]^-{\tau_1}\ar[dr]_-{\ol\delta_1\delta_1}&(R_1,\m_1,k_1)\ar[rd]_-{\sigma_1\beta_1}&(R,\m,k)\ar[r]^-{\vf_2}\ar[l]_-{\vf_1}\ar[d]^-{\alpha'}
&(R_2,\m_2,k_2)\ar[ld]^-{\sigma_2\beta_2}&(Q_2,\n_2,k_2)\ar[l]_-{\tau_2}\ar[dl]^-{\ol\delta_2\delta_2}\\
&(\overline{Q}_1,\ol{\n}_1,k')\ar[r]_-{\ol\gamma_1}&(R',\m',k')&(\overline{Q}_2,\ol{\n}_1,k')\ar[l]^-{\ol\gamma_2}& \hspace{11mm}\text{(C.1)}
}
$$
commutes where $\alpha'=\sigma_1\beta_1\vf_1=\sigma_2\beta_2\vf_2$.
By assumption, each $\ol\gamma_i$ is surjective.
Since $\sigma_1$, $\beta_1$, and $\vf_1$ are flat, so is their composition $\alpha'$.
The maps $\ol\delta_i$ and $\delta_i$ are weakly regular, hence Fact~\ref{fact120215a}
implies that their composition is weakly regular. Moreover, since the field
extension $k_i\to k'$ is separable, the composition $\ol\delta_i\delta_i$
is weakly Cohen.

\eqref{intthm120201d02}
Assume that each map $\vf_i$ is weakly Cohen. Then the closed fiber $R_i/\m R_i$ is regular. Since $\beta_i\colon R_i\to\wti R_i$ is Cohen, the same is true of the induced
map $R_i/\m R_i\to \wti R_i/\m \wti R_i$.
Thus, the fact that $R_i/\m R_i$ is regular implies that $\wti R_i/\m \wti R_i$ is also regular.
Since $\alpha\colon R\to \wti R$ is Cohen, we have $\m\wti R=\wti \m$,
and it follows that $\wti R_i/\m \wti R_i=\wti R_i/\wti\m \wti R_i$ is regular.
From~\cite[Lemme 19.7.1.2]{grothendieck:ega4-1} we know  that the closed fiber of the  map
$\sigma_1\colon(\wti R_1,\wti\m_1,k')\to (R',\m',k')$
is 
$$R'/\wti\m_1R'
\cong [\wti R_1/\wti \m_1]\otimes_{\wti R}\wti R_2
\cong[\wti R_1/\wti \m_1]\otimes_{k'}[\wti R_2/\wti\m \wti R_2]
\cong k'\otimes_{k'}[\wti R_2/\wti\m \wti R_2]\cong \wti R_2/\wti\m \wti R_2.
$$
Since this ring is regular
and the field extension $k'\to k'$ is trivially separable, the map $\sigma_1$ is 
weakly Cohen, as is $\sigma_2$ by similar argument.
Thus, Fact~\ref{fact120418a}\eqref{fact120418a2}
implies that diagram~\eqref{eq120409b}
is a pushout. As in an  earlier part of this proof, for each $i=1, 2$  the diagram 
$R\xra{\sigma_i\beta_i\vf_i} R'\xla{\ol\gamma_i} \overline{Q}_i$ is a quasi-deformation.

By construction, the maps $R_i\xra{\sigma_i\beta_i}R'$ are compositions
of weakly Cohen maps, so they are weakly Cohen; hence 
conclusion~\eqref{intthm120201d1} from the statement of 
Theorem~\ref{intthm120201d} is satisfied.
Since the diagrams~\eqref{eq120409a} and~\eqref{eq120409b} are pushouts,
the same is true of the parallelograms in the diagram above; hence 
conclusion~\eqref{intthm120201d2} from the statement of 
Theorem~\ref{intthm120201d} is satisfied.
Lastly, given an $R$-module $M$,
the flatness of the map $Q_i\xra{\ol\delta_i\delta_i}\ol Q_i$ provides
the first equality in the next display:
\begin{align*}
\pd_{Q_i}(M\otimes_RR_i)
&=\pd_{\ol Q_i}((M\otimes_RR_i)\otimes_{Q_i}\ol Q_i)\\
&=\pd_{\ol Q_i}(M\otimes_R(R_i\otimes_{Q_i}\ol Q_i))\\
&=\pd_{\ol Q_i}(M\otimes_RR').
\end{align*}
(See~\cite[Theorem 9.6]{perry:ffdpm} and~\cite{raynaud:cpptpm}.)
The last equality is from the pushout conclusion on each parallelogram.
This shows that
conclusion~\eqref{intthm120201d4} from the statement of 
Theorem~\ref{intthm120201d} is satisfied.
\qed
\end{para}

\begin{disc}
At this time, we do not know if the  
weakly Cohen assumption is  necessary in Theorem~\ref{intthm120201d}.
\end{disc}

\section*{Acknowledgments}
We are grateful to Javier Majadas, Tirdad Sharif, and the referee for their thoughtful and useful comments.


\begin{thebibliography}{10}

\bibitem{auslander:adgeteac}
M.~Auslander, \emph{Anneaux de {G}orenstein, et torsion en alg\`ebre
  commutative}, S\'eminaire d'Alg\`ebre Commutative dirig\'e par Pierre Samuel,
  vol. 1966/67, Secr\'etariat math\'ematique, Paris, 1967. \MR{37 \#1435}

\bibitem{auslander:smt}
M.~Auslander and M.\ Bridger, \emph{Stable module theory}, Memoirs of the
  American Mathematical Society, No. 94, American Mathematical Society,
  Providence, R.I., 1969. \MR{42 \#4580}

\bibitem{avramov:lcih}
L.~L. Avramov, \emph{Locally complete intersection homomorphisms and a
  conjecture of {Q}uillen on the vanishing of cotangent homology}, Ann. of
  Math. (2) \textbf{150} (1999), no.~2, 455--487. \MR{1726700 (2001a:13024)}

\bibitem{avramov:glh}
L.~L. Avramov and H.-B.\ Foxby, \emph{Gorenstein local homomorphisms}, Bull.
  Amer. Math. Soc. (N.S.) \textbf{23} (1990), no.~1, 145--150. \MR{1020605
  (90k:13009)}

\bibitem{avramov:hdouc}
\bysame, \emph{Homological dimensions of unbounded complexes}, J. Pure Appl.
  Algebra \textbf{71} (1991), 129--155. \MR{93g:18017}

\bibitem{avramov:lgh}
\bysame, \emph{Locally {G}orenstein homomorphisms}, Amer. J. Math. \textbf{114}
  (1992), no.~5, 1007--1047. \MR{1183530 (93i:13019)}

\bibitem{avramov:rhafgd}
\bysame, \emph{Ring homomorphisms and finite {G}orenstein dimension}, Proc.
  London Math. Soc. (3) \textbf{75} (1997), no.~2, 241--270. \MR{98d:13014}

\bibitem{avramov:cmporh}
\bysame, \emph{Cohen-{M}acaulay properties of ring homomorphisms}, Adv. Math.
  \textbf{133} (1998), no.~1, 54--95. \MR{1492786 (99c:13043)}

\bibitem{avramov:solh}
L.~L. Avramov, H.-B.\ Foxby, and B.\ Herzog, \emph{Structure of local
  homomorphisms}, J. Algebra \textbf{164} (1994), 124--145. \MR{95f:13029}

\bibitem{avramov:cid}
L.~L. Avramov, V.~N. Gasharov, and I.~V. Peeva, \emph{Complete intersection
  dimension}, Inst. Hautes \'Etudes Sci. Publ. Math. (1997), no.~86, 67--114
  (1998). \MR{1608565 (99c:13033)}

\bibitem{christensen:gd}
L.~W. Christensen, \emph{Gorenstein dimensions}, Lecture Notes in Mathematics,
  vol. 1747, Springer-Verlag, Berlin, 2000. \MR{2002e:13032}

\bibitem{foxby:hacr}
H.-B.\ Foxby, \emph{Hyperhomological algebra \& commutative rings}, lecture
  notes.

\bibitem{frankild:qcmpolh}
A.~Frankild, \emph{Quasi {C}ohen-{M}acaulay properties of local homomorphisms},
  J. Algebra \textbf{235} (2001), 214--242. \MR{2001j:13023}

\bibitem{gelfand:moha}
S.~I. Gelfand and Y.~I. Manin, \emph{Methods of homological algebra},
  Springer-Verlag, Berlin, 1996. \MR{2003m:18001}

\bibitem{gerko:ohd}
A.~A. Gerko, \emph{On homological dimensions}, Mat. Sb. \textbf{192} (2001),
  no.~8, 79--94, translation in Sb.\ Math. \textbf{192} (2001), no.\ 7--8,
  1165--1179. \MR{2002h:13024}

\bibitem{grothendieck:ega1}
A.~Grothendieck, \emph{\'{E}l\'ements de g\'eom\'etrie alg\'ebrique. {I}. {L}e
  langage des sch\'emas}, Inst. Hautes \'Etudes Sci. Publ. Math. (1960), no.~4,
  228. \MR{0217083 (36 \#177a)}

\bibitem{grothendieck:ega3-1}
A.\ Grothendieck, \emph{\'{E}l\'ements de g\'eom\'etrie alg\'ebrique. {III}.
  \'{E}tude cohomologique des faisceaux coh\'erents. {I}}, Inst. Hautes
  \'Etudes Sci. Publ. Math. (1961), no.~11, 167. \MR{0217085 (36 \#177c)}

\bibitem{grothendieck:ega4-1}
\bysame, \emph{\'{E}l\'ements de g\'eom\'etrie alg\'ebrique. {IV}. \'{E}tude
  locale des sch\'emas et des morphismes de sch\'emas. {I}}, Inst. Hautes
  \'Etudes Sci. Publ. Math. (1964), no.~20, 259.

\bibitem{grothendieck:ega4-4}
\bysame, \emph{\'{E}l\'ements de g\'eom\'etrie alg\'ebrique. {IV}. \'{E}tude
  locale des sch\'emas et des morphismes de sch\'emas {IV}}, Inst. Hautes
  \'Etudes Sci. Publ. Math. (1967), no.~32, 361.

\bibitem{hartshorne:rad}
R.\ Hartshorne, \emph{Residues and duality}, Lecture Notes in Mathematics, No.
  20, Springer-Verlag, Berlin, 1966. \MR{36 \#5145}

\bibitem{herzog:mlr}
B.\ Herzog, \emph{On the macaulayfication of local rings}, J. Algebra
  \textbf{67} (1980), no.~2, 305--317. \MR{0602065 (82c:13029)}

\bibitem{iyengar:golh}
S.\ Iyengar and S.\ Sather-Wagstaff, \emph{G-dimension over local
  homomorphisms. {A}pplications to the {F}robenius endomorphism}, Illinois J.
  Math. \textbf{48} (2004), no.~1, 241--272. \MR{2048224 (2005c:13016)}

\bibitem{kawasaki:mns}
T.~Kawasaki, \emph{On {M}acaulayfication of {N}oetherian schemes}, Trans. Amer.
  Math. Soc. \textbf{352} (2000), no.~6, 2517--2552. \MR{1707481 (2000j:14077)}

\bibitem{matsumura:crt}
H.\ Matsumura, \emph{Commutative ring theory}, second ed., Studies in Advanced
  Mathematics, vol.~8, University Press, Cambridge, 1989. \MR{90i:13001}

\bibitem{perry:ffdpm}
A.~Perry, \emph{Faithfully flat descent for projectivity of modules}, preprint
  (2012) \texttt{arXiv:1011.0038v1}.

\bibitem{raynaud:cpptpm}
M.\ Raynaud and L.\ Gruson, \emph{Crit\`eres de platitude et de projectivit\'e.
  {T}echniques de ``platification'' d'un module}, Invent. Math. \textbf{13}
  (1971), 1--89. \MR{0308104 (46 \#7219)}

\bibitem{sather:cidc}
S.\ Sather-Wagstaff, \emph{Complete intersection dimensions for complexes}, J.
  Pure Appl. Algebra \textbf{190} (2004), no.~1-3, 267--290. \MR{2043332
  (2005i:13022)}

\bibitem{sather:cidfc}
\bysame, \emph{Complete intersection dimensions and {F}oxby classes}, J. Pure
  Appl. Algebra \textbf{212} (2008), no.~12, 2594--2611. \MR{2452313
  (2009h:13015)}

\bibitem{serre:alm}
J.-P. Serre, \emph{Alg\`ebre locale. {M}ultiplicit\'es}, Cours au Coll\`ege de
  France, 1957--1958, r\'edig\'e par Pierre Gabriel. Seconde \'edition, 1965.
  Lecture Notes in Mathematics, vol.~11, Springer-Verlag, Berlin, 1965.
  \MR{0201468 (34 \#1352)}

\bibitem{verdier:cd}
J.-L.\ Verdier, \emph{Cat\'{e}gories d\'{e}riv\'{e}es}, SGA 4$\frac{1}{2}$,
  Springer-Verlag, Berlin, 1977, Lecture Notes in Mathematics, Vol. 569,
  pp.~262--311. \MR{57 \#3132}

\bibitem{verdier:1}
\bysame, \emph{Des cat\'egories d\'eriv\'ees des cat\'egories ab\'eliennes},
  Ast\'erisque (1996), no.~239, xii+253 pp. (1997), With a preface by Luc
  Illusie, Edited and with a note by Georges Maltsiniotis. \MR{98c:18007}

\bibitem{yassemi:gd}
S.\ Yassemi, \emph{G-dimension}, Math. Scand. \textbf{77} (1995), no.~2,
  161--174. \MR{97d:13017}

\end{thebibliography}
\providecommand{\bysame}{\leavevmode\hbox to3em{\hrulefill}\thinspace}
\providecommand{\MR}{\relax\ifhmode\unskip\space\fi MR }
\providecommand{\MRhref}[2]{%
  \href{http://www.ams.org/mathscinet-getitem?mr=#1}{#2}
}
\providecommand{\href}[2]{#2}

\end{document}